\documentclass{article}

\usepackage{amssymb,amsmath,amsthm}

\newcommand{\N}{\mathbb{N}}
\newcommand{\R}{\mathbb{R}}
\newcommand{\B}{\mathcal{B}}
\newcommand{\K}{\mathcal{K}}
\newcommand{\D}{\mathcal{D}}

\newcommand{\ga}{\gamma}
\newcommand{\Ga}{\Gamma}
\newcommand{\la}{\lambda}
\newcommand{\La}{\Lambda}
\newcommand{\eps}{\varepsilon}
\newcommand{\1}{1\!\!1}
\newcommand\ren{{\eps, \, \mathrm{ren}}}
\newcommand{\X}{{\R^d}}

\DeclareMathOperator*{\esssup}{ess\,sup}
\newcommand{\inv}{\mathrm{inv}}
\newcommand{\Dom}{\mathrm{Dom}}
\newcommand{\Sect}{\mathrm{Sect}}

\renewcommand{\Re}{\mathrm{Re\,}}
\renewcommand{\Im}{\mathrm{Im\,}}

\numberwithin{equation}{section}

\newtheorem{theorem}{Theorem}[section]
\newtheorem{lemma}[theorem]{Lemma}

\newtheorem{proposition}[theorem]{Proposition}
\theoremstyle{remark}
\newtheorem{remark}{Remark}[section]
\theoremstyle{definition}
\newtheorem{assumption}{Assumption}[section]
\newtheorem{example}{Example}

\allowdisplaybreaks[3]

\makeatletter
\renewcommand{\@fnsymbol}[1]{\ensuremath{%
   \ifcase#1\or * \or 1\or 2\or 3\or
   \mathsection\or \mathparagraph\or \|\or \star\or
   \star\star\or {\star\star}\star \else\@ctrerr\fi}}
\makeatother

\title{Semigroup approach to birth-and-death stochastic dynamics in~continuum}

\author{Dmitri Finkelshtein\thanks{Institute of Mathematics,
National Academy of Sciences of Ukraine, Kyiv, Ukraine ({\tt
fdl@imath.kiev.ua}).} \and Yuri Kondratiev\thanks{Fakult\"{a}t
f\"{u}r Mathematik, Universit\"{a}t Bielefeld, 33615 Bielefeld,
Germany ({\tt kondrat@math.uni-bielefeld.de})} \and Oleksandr
Kutoviy\thanks{Fakult\"{a}t f\"{u}r Mathematik, Universit\"{a}t
Bielefeld, 33615 Bielefeld, Germany ({\tt
kutoviy@math.uni-bielefeld.de}).}}

\begin{document}

\maketitle

\begin{abstract}
We describe a general approach to the construction of a state
evolution corresponding to the Markov generator of a spatial
birth-and-death dynamics in $\X$. We present conditions on the
birth-and-death intensities which are sufficient for the existence
of an evolution as a strongly continuous semigroup in a proper
Banach space of correlation functions satisfying the Ruelle bound.
The convergence of a Vlasov-type scaling for the corresponding
stochastic dynamics is considered.
\end{abstract}



\section{Introduction}
Spatial Markov processes in $\X$ may be described as stochastic
evolutions of locally finite subsets (configurations)
$\ga\subset\X$, i.e., any $\ga$ has a finite number of points in an
arbitrary ball in $\X$. One of the most important classes of such
stochastic dynamics is given by the birth-and-death Markov processes
in the space $\Gamma$ of all configurations from $\X$. These are
processes in which an infinite number of individuals exist at each
instant, and the rates at which new individuals appear and some old ones
disappear depend on the instantaneous configuration of existing
individuals \cite{HS1978}. The corresponding Markov generators have
a~natural heuristic representation in terms of birth and death
intensities. The birth intensity $b(x,\ga)\geq0$ characterizes the
appearance of a~new point at $x\in\X$ in~the presence of a~given
configuration $\ga\in\Ga$. The death intensity $d(x,\ga)\geq0$
characterizes the probability of the event that the point $x$ of the configuration
$\ga$ disappears, depending on the location of the remaining
points of the confuguration, $\ga\setminus x$. Here and below, for
simplicity of notation, we write $x$ instead of $\{x\}$.
Heuristically, the corresponding Markov generator is described by
the following expression
\begin{align}
(LF)(\ga ):=&\sum_{x\in \ga }d(x,\ga \setminus x)\left[F(\ga
\setminus x)-F(\ga )\right]\notag\\&+\int_{\R^{d}}b(x,\ga
)\left[F(\ga \cup x)-F(\ga )\right] dx, \label{BaDGen-intro}
\end{align}
for proper functions $F:\Ga\rightarrow\R$.

The study of spatial birth-and-death processes was initiated by
C. Preston \cite{Pre1975}. This paper dealt with a solution of the
backward Kolmogorov equation
\begin{equation}\label{BKE}
  \frac{\partial}{\partial t} F_t = L F_t
\end{equation}
under the restriction that only a~finite number of individuals are
alive at each moment of time. Under certain conditions,
corresponding processes exist and are temporally ergodic, that is,
there exists a~unique stationary distribution. Note that a more general
setting for birth-and-death processes only requires that the number of points in any compact
set remains finite at all times. A further progress in the study of these
processes was achieved by R. Holley and D. Stroock in \cite{HS1978}. They
described in detail an analytic framework for birth-and-death
dynamics. In particular, they analyzed the case of a birth-and-death process in
a bounded region.

Stochastic equations for spatial birth-and-death processes were
formulated in \cite{Gar1995}, through a~spatial version of the
time-change approach. Further, in \cite{GK2006}, these processes were
represented as solutions to a~system of stochastic equations, and
conditions for the existence and uniqueness of solutions to these
equations, as well as for the corresponding martingale problems, were
given. Unfortunately, quite restrictive assumptions on the birth and
death rates in \cite{GK2006} do not allow an application of
these results to several particular models that are interesting for
applications (see e.g. Examples~1--3 below).

A growing interest to the study of spatial birth-and-death processes, which we have recently observed, is stimulated
by (among others) an
important role which these processes play in several applications. For example,
in spatial plant ecology, a~general approach to the so-called
individual based models was developed in a~series of works, see e.g.\
\cite{BP1997,BP1999,DL2000,MDL2004} and the references therein. These
models are described as birth-and-death Markov processes in the
configuration space $\Gamma$ with specific rates $b$ and $d$
which reflect biological notions such as competition, establishment,
fecundity etc. Other examples of birth-and-death processes may be
found in mathematical physics. In particular, the Glauber-type
stochastic dynamics in $\Gamma$ is properly associated with the grand
canonical Gibbs measures for classical gases. This gives a~possibility to
study these Gibbs measures as equilibrium states for specific
birth-and-death Markov evolutions \cite{BCC2002}. Starting with a~Dirichlet form for a~given Gibbs measure, one
can consider an equilibrium stochastic dynamics \cite{KL2005}.
However, these dynamics give the time evolution of initial distributions from a quite narrow class. Namely, the class of admissible initial
distributions is essentially reduced to the states which are absolutely
continuous with respect to the invariant measure. In the present paper we
construct non-equilibrium stochastic dynamics which may have a much
wider class of initial states.

Concerning the study of particular birth-and-death models, let us stress that, on the one hand, for most cases appearing in
applications, the existence problem for a corresponding Markov
process is still open. On the other hand, the evolution of a state in the course of a stochastic dynamics is an important question in its own right.
A~mathematical formulation of this question may be realized through
the forward Kolmogorov equation for probability measures (states) on
the configuration space $\Gamma$:
\begin{equation}\label{FPE}
  \frac{\partial}{\partial t} \mu_t = L^* \mu_t.
\end{equation}
Here $L^*$ is the (informally) adjoint operator of $L$ with respect to the
pairing
\begin{equation}\label{pairing}
  \langle F,\mu \rangle:=\int_\Ga F(\ga)\,d\mu(\ga).
\end{equation}
In the physical literature, \eqref{FPE} is known as the Fokker--Planck
equation. However, the mere existence of the corresponding Markov
process will not give us much information about properties
of the solution to \eqref{FPE}.

An important technical observation concerns the~possibility to
reformulate the equations for states in terms of time evolutions for
corresponding correlation functions, see e.g. \cite{FKO2009} and
references therein. Namely, a~probability measure $\mu$ on $\Ga$ may
be characterized by a~sequence
$\bigl\{k^{(n)}(x_1,\ldots,x_n)\bigr\}_{n=0}^\infty$ of symmetric
non-negative functions on $(\X)^n$. Then, \eqref{FPE} may be
rewritten in the form
\begin{equation}\label{QE}
  \frac{\partial}{\partial t} k_t = \hat{L}^* k_t,
\end{equation}
where $\hat{L}^*$ is the~corresponding image of the operator $L^*$
from \eqref{FPE} acting on sequences of functions
$k_t=\{k^{(n)}_t\}_{n=0}^\infty$.

In various applications, correlation functions satisfy the so-called
Ruelle bound
\begin{equation}\label{RB-intro}
  \lvert k^{(n)}(x_1,\ldots,x_n)\rvert \leq C^n,\quad x_1,\ldots,x_n\in\X, \ n\in\N
\end{equation}
for some $C>0$. For example, for the correlation functions of the
Gibbs measure mentioned above, such inequalities hold true, see e.g.
\cite{Rue1970}. Hence, it is rather natural to study the solutions to
the equation \eqref{QE} in weighted $L^\infty$-type space of
functions with the Ruelle bound. However, analysis of the existence
problem in such a class of correlation functions meats essential
difficulties related to the use of non-separable $L^\infty$ spaces
and properties of strongly continuous semigroups acting in these
spaces. One of technical possibilities to study such
semigroups is based on the use of the pre-dual evolution equations in some
$L^1$ spaces.

Namely, we will exploit the duality
\begin{equation}\label{duality-intro}
  \langle\!\langle G,k \rangle\!\rangle:=
  \sum_{n=0}^\infty \frac{1}{n!} \int_{(\X)^n}
  G^{(n)}(x_1,\ldots,x_n)
  k^{(n)}(x_1,\ldots,x_n)\,dx_1\ldots dx_n,
\end{equation}
which is a pairing between a sequence $k=\{k^{(n)}\}_{n=0}^\infty$ of
functions which satisfy \eqref{RB-intro} and a sequence
$G=\{G^{(n)}\}_{n=0}^\infty$ of the so-called quasi-observables. The latter
are integrable functions satisfying
\begin{equation}\label{LCnorm-intro}
  \sum_{n=0}^\infty \frac{C^n}{n!} \int_{(\X)^n}\bigl|
G^{(n)}(x_1,\ldots,x_n)\bigr|\,dx_1\ldots dx_n<\infty.
\end{equation}
Then, the equation \eqref{QE} may be rewritten as follows
\begin{equation}\label{QKE}
  \frac{\partial}{\partial t} G_t = \hat{L} G_t,
\end{equation}
with the corresponding operator $\hat{L}$ acting on sequences
$G_t=\{G^{(n)}_t\}_{n=0}^\infty$. This is an analog of the backward
Kolmogorov equation \eqref{BKE} on sequences of functions. Note that
$\hat{L}^*$ is the dual operator of $\hat{L}$ with respect to
the duality \eqref{duality-intro}. The resulting, so-called hierarchical
equation \eqref{QKE} may be analyzed in a Fock-type space of
sequences of functions which satisfy \eqref{LCnorm-intro}. The
corresponding semigroup may be used for a~construction of time
evolution \eqref{QE} for correlation functions using the duality
\eqref{duality-intro}.

This approach was successfully applied to the construction and
analysis of state evolutions for different versions of the
Glauber dynamics \cite{KKZ2006,FKKZ2010,FKK2010} and for some
spatial ecology models \cite{FKK2009}. Each of the considered models
required its own specific version of the construction of a semigroup, which
takes into account particular properties of corresponding birth and
death rates.

In the present paper, we develop a~general approach to the
construction of the state evolution corresponding to the
birth-and-death Markov generators. We present conditions on the birth
and death intensities which are sufficient for the existence of
corresponding evolutions as strongly continuous semigroups in proper
Banach spaces of correlation functions satisfying the Ruelle bound
\eqref{RB-intro}.

Moreover, we apply this construction to study of the convergence of
the considered stochastic dynamics in a Vlasov-type scaling.
Originally, the notion of the Vlasov scaling was related to the
Hamiltonian dynamics of interacting particle systems. This is a mean
field scaling limit when the influence of weak long-range forces is
taken into account. Rigorously, this limit was studied by W.~Braun
and K.~Hepp in \cite{BH1977} for the Hamiltonian dynamics, and by
R.L.~Dobrushin \cite{Dob1979} for more general deterministic
dynamical systems. In \cite{FKK2010a}, we proposed a general scheme
for a Vlasov-type scaling of stochastic Markovian dynamics. Our
approach is based on a proper scaling of the evolutions of
correlation functions proposed by H. Spohn in \cite{Spo1980} for the
Hamiltonian dynamics. In the present paper, we apply such an
approach to the birth-and-death stochastic dynamics. This gives us
a~rigorous framework for the study of convergence of the scaled
hierarchical equations to a solution of the limiting Vlasov
hierarchy, and for the derivation of a~resulting non-linear
evolutional equation for the density of the limiting system. We
consider some special birth-and-death models to show how the general
conditions proposed in the paper may be verified in applications.

The structure of the paper is as follows. In
Section~\ref{sect-Prelim} we give a~brief introduction to notions related to the
configuration space. Subsection~\ref{subsect-evol-qo} is
devoted to the evolution of quasi-observables in the Fock-type space
which is the pre-dual of the space of correlation functions. We
propose constructive conditions on the birth and death rates under
which the corresponding dynamics exist. These conditions are
verified for a number of particular examples. The evolution of
correlation functions is considered in
Subsection~\ref{subsect-evol-cf}. The question concerning the existence
and uniqueness of the solution to the corresponding stationary
equation in the space of correlation functions is studied in
Subsection~\ref{subsect-evol-se}. In~Section~\ref{sect-VS} we
discuss the Vlasov-type scaling for birth-and-death stochastic
dynamics.

\section{Basic facts and notation}\label{sect-Prelim}

Let ${\B}({\X})$ be the family of all Borel sets in ${\X}$, $d\geq
1$; ${\B}_{\mathrm{b}} ({\X})$ denotes the system of all bounded
sets from ${\B}({\X})$.

The configuration space over space $\X$ consists of all locally
finite subsets (configurations) of $\X$. Namely,
\begin{equation} \label{confspace}
\Ga =\Ga\bigl(\X\bigr) :=\Bigl\{ \ga \subset \X \Bigm| |\ga _\La
|<\infty, \ \mathrm{for \ all } \ \La \in {\B}_{\mathrm{b}}
(\X)\Bigr\}.
\end{equation}
Here $|\cdot|$ means the cardinality of a~set, and
$\ga_\La:=\ga\cap\La$. The space $\Ga$ is equipped with the vague
topology, i.e., the weakest topology for which all mappings
$\Ga\ni\ga\mapsto \sum_{x\in\ga} f(x)\in{\R}$ are continuous for any
continuous function $f$ on $\X$ with compact support. The
corresponding Borel $\sigma $-algebra $\B(\Ga )$ is the smallest
$\sigma $-algebra for which all mappings $\Ga \ni \ga \mapsto |\ga_
\La |\in{ \N}_0:={\N}\cup\{0\}$ are measurable for any $\La\in{
\B}_{\mathrm{b}}(\X)$, see e.g. \cite{AKR1998a}. It is worth noting
that $\Ga$ is a~Polish space (see e.g. \cite{KK2006} and references
therein).

The space of $n$-point configurations in $Y\in\B(\X)$ is defined by
\begin{equation*}
\Ga^{(n)}(Y):=\Bigl\{ \eta \subset Y \Bigm| |\eta |=n\Bigr\} ,\quad
n\in { \N}.
\end{equation*}
We set $\Ga^{(0)}(Y):=\{\emptyset\}$. As a~set, $\Ga^{(n)}(Y)$ may
be identified with the symmetrization of $\widetilde{Y^n} = \bigl\{
(x_1,\ldots ,x_n)\in Y^n \bigm| x_k\neq x_l \ \mathrm{if} \ k\neq
l\bigr\}$. Hence one can introduce the corresponding Borel $\sigma
$-algebra, which we denote by $\B\bigl(\Ga^{(n)}(Y)\bigr)$. The
space of finite configurations in $Y\in\B(\X)$ is defined as
\begin{equation*}
\Ga_0(Y):=\bigsqcup_{n\in {\N}_0}\Ga^{(n)}(Y).
\end{equation*}
This space is equipped with the topology of the disjoint union. Let
$\B \bigl(\Ga_0(Y)\bigr)$ denote the corresponding Borel $\sigma
$-algebra. In the case of $Y=\X$ we will omit the index $Y$ in the
previously defined notations. Namely, $\Ga_0:=\Ga_{0}(\X)$,
$\Ga^{(n)}:=\Ga^{(n)}(\X)$.

The restriction of the Lebesgue product measure $(dx)^n$ to
$\bigl(\Ga^{(n)}, \B(\Ga^{(n)})\bigr)$ we denote by $m^{(n)}$. We
set $m^{(0)}:=\delta_{\{\emptyset\}}$. The Lebesgue--Poisson measure
$\la $ on $\Ga_0$ is defined by
\begin{equation} \label{LP-meas-def}
\la :=\sum_{n=0}^\infty \frac {1}{n!}m^{(n)}.
\end{equation}
For any $\La\in\B_{\mathrm{b}}(\X)$ the restriction of $\la$ to $\Ga
(\La):=\Ga_{0}(\La)$ will be also denoted by $\la $. The space
$\bigl( \Ga, \B(\Ga)\bigr)$ is the projective limit of the family of
spaces $\bigl\{\bigl( \Ga(\La), \B(\Ga(\La))\bigr)\bigr\}_{\La \in
\B_{\mathrm{b}} (\X)}$. The Poisson measure $\pi$ on $\bigl(\Ga
,\B(\Ga )\bigr)$ is given as the projective limit of the family of
measures $\{\pi^\La \}_{\La \in \B_{\mathrm{b}} (\X)}$, where $
\pi^\La:=e^{-m(\La)}\la $ is the probability measure on $\bigl(
\Ga(\La), \B(\Ga(\La))\bigr)$ and $m(\La)$ is the Lebesgue measure
of $\La\in \B_{\mathrm{b}} (\X)$ (see e.g. \cite{AKR1998a} for
details).

A set $M\in \B (\Ga_0)$ is called bounded if there exists $ \La \in
\B_{\mathrm{b}} (\X)$ and $N\in { \N}$ such that $M\subset
\bigsqcup_{n=0}^N\Ga^{(n)}(\La)$. The set of bounded measurable
functions with bounded support we denote by $
B_{\mathrm{bs}}(\Ga_0)$, i.e., $G\in B_{\mathrm{bs}}(\Ga_0)$ if $
G\upharpoonright_{\Ga_0\setminus M}=0$ for some bounded $M\in {\B
}(\Ga_0)$. Any $\B(\Ga_0)$-measurable function $G$ on $ \Ga_0$, in
fact, is defined by a~sequence of functions
$\bigl\{G^{(n)}\bigr\}_{n\in{ \N}_0}$ where $G^{(n)}$ is a
$\B(\Ga^{(n)})$-measurable function on $\Ga^{(n)}$. The set of
\textit{cylinder functions} on $\Ga$ we denote by ${{\mathcal{
F}}_{\mathrm{cyl}}}(\Ga )$. Each $F\in
{{\mathcal{F}}_{\mathrm{cyl}}}(\Ga )$ is characterized by the
following relation: $F(\ga )=F(\ga_\La )$ for some $\La\in
\B_{\mathrm{b}}(\X)$. Functions on $\Ga$ will be called {\em
observables} whereas functions on $\Ga_0$ well be called {\em
quasi-observables}.

There exists mapping from $B_{\mathrm{bs}} (\Ga_0)$ into ${{
\mathcal{F}}_{\mathrm{cyl}}}(\Ga )$, which plays the key role in our
further considerations:
\begin{equation}
(KG)(\ga ):=\sum_{\eta \Subset \ga }G(\eta ), \quad \ga \in \Ga,
\label{KT3.15}
\end{equation}
where $G\in B_{\mathrm{bs}}(\Ga_0)$, see e.g.
\cite{KK2002,Len1975,Len1975a}. The summation in \eqref{KT3.15} is
taken over all finite subconfigurations $\eta\in\Ga_0$ of the
(infinite) configuration $\ga\in\Ga$; we denote this by the symbol,
$\eta\Subset\ga $. The mapping $K$ is linear, positivity preserving,
and invertible, with
\begin{equation}
(K^{-1}F)(\eta ):=\sum_{\xi \subset \eta }(-1)^{|\eta \setminus \xi
|}F(\xi ),\quad \eta \in \Ga_0. \label{k-1trans}
\end{equation}
Set $(K_0 G)(\eta):=(KG)(\eta)$, $\eta\in\Ga_0$.

The so-called coherent state corresponding to a~$\B(\X)$-measurable
function $f$ is defined by
\[
e_\la (f,\eta ):=\prod_{x\in \eta }f(x) ,\ \eta \in \Ga
_0\!\setminus\!\{\emptyset\},\quad e_\la (f,\emptyset ):=1.
\]
Then
\begin{equation}\label{Kexp}
(K_0e_\la (f))(\eta)=e_\la(f+1,\eta), \quad \eta\in\Ga_0
\end{equation}
and for any $f\in L^1(\X,dx)$
\begin{equation}\label{intexp}
\int_{\Ga_0}e_\la (f,\eta)d\la(\eta)=\exp\Bigl\{\int_\X
f(x)dx\Bigr\}.
\end{equation}

A measure $\mu \in {\mathcal{M}}_{\mathrm{fm} }^1(\Ga )$ is called
locally absolutely continuous with respect to the Poisson measure
$\pi$ if for any $\La \in \B_{\mathrm{b}} (\X)$ the projection of
$\mu$ onto $\Ga(\La)$ is absolutely continuous with respect to the
projection of $ \pi$ onto $\Ga(\La)$. In this case, according to
\cite{KK2002}, there exists a~\emph{correlation functional}
$k_{\mu}:\Ga_0 \rightarrow {\R}_+$ such that for any $G\in
B_{\mathrm{bs}} (\Ga_0)$ the following equality holds
\begin{equation} \label{eqmeans}
\int_\Ga (KG)(\ga) d\mu(\ga)=\int_{\Ga_0}G(\eta)
k_\mu(\eta)d\la(\eta).
\end{equation}
The functions $ k_{\mu}^{(n)}:(\R^{d})^{n}\longrightarrow\R_{+} $
given by
\[
k_{\mu}^{(n)}(x_{1},\ldots,x_{n}):=
\begin{cases}
k_{\mu}(\{x_{1},\ldots,x_{n}\}), & \mathrm{if} \
(x_{1},\ldots,x_{n})\in \widetilde{(\R^{d})^{n}}\\
0, & \mathrm{ otherwise}
\end{cases}
\]
are called \emph{correlation functions} of the measure $\mu$. Note
that $k_\mu^{(0)}=1$.

Below we would like to mention without proof the partial case of the
well-known technical lemma (see e.g. \cite{KMZ2004}) which plays
very important role in our calculations.

\begin{lemma}
\label{Minlos} For any measurable function $H:\Ga_0\times\Ga_0\times
\Ga_0\rightarrow{\R}$
\begin{equation} \label{minlosid}
\int_{\Ga _{0}}\sum_{\xi \subset \eta }H\left( \xi ,\eta \setminus
\xi ,\eta \right) d\la \left( \eta \right) =\int_{\Ga _{0}}\int_{\Ga
_{0}}H\left( \xi ,\eta ,\eta \cup \xi \right) d\la \left( \xi
\right) d\la \left( \eta \right)
\end{equation}
if both sides of the equality make sense.
\end{lemma}

\section{Non-equilibrium evolutions}\label{sect-NonEq-evol}

In a birth-and-death dynamics, particles appear and disappear randomly
in $\R^{d}$ according to birth and death rates which depend on the
configuration of the whole system. Heuristically, the corresponding
Markov generator is described by the following expression
\begin{align}
(LF)(\ga ):=&\sum_{x\in \ga }d(x,\ga \setminus x)\left[F(\ga
\setminus x)-F(\ga )\right]\notag\\&+\int_{\R^{d}}b(x,\ga
)\left[F(\ga \cup x)-F(\ga )\right] dx. \label{BaDGen}
\end{align}
Here the coefficient $d(x,\ga )\geq 0$ represents the rate at which
particle of the configuration $\ga$ located at $x$ dies
(disappears), whereas, for a given configuration $\ga $, the new
particle appears at the site $x$ with the rate $b(x,\ga )\geq 0$.

We always suppose that, for all $x\in\X$ and a.a. $x'\in\X$, the
values $d(x,\eta)$ and $b(x',\eta)$ are finite at least for all
configurations $\eta\in\Ga_0$ which do not contain the points $x$
and $x'$. Here and below, we assume that, for a.a. $x\in\X$,
the functions $d(x,\cdot)$ and $b(x,\cdot)$ are locally integrable,
i.e., for all bounded $M\in\B(\Ga_0)$
\[
\int_M \bigl( d(x,\eta)+ b(x,\eta)\bigr) d\la(\eta)<\infty.
\]

A natural way to study a Markov evolution with generator
\eqref{BaDGen} is to construct a corresponding Markov semigroup
with the generator $L$. This problem is related to the analysis of
the initial value problem
\begin{equation}\label{BKE-init}
  \frac{\partial}{\partial t} F_t = L F_t, \quad t>0, \quad
  F_t\bigr|_{t=0}=F_0
\end{equation}
in some space of functions on the configuration space
$\Ga$.
However, a rigorous analysis of such evolutional equations meets serious
technical problems, and was realized for the case of birth and death
generator in a finite volumes only, see \cite{HS1978}. On the other
hand, there is a very important question concerning the state evolution
associated with Markov dynamics. Namely, one can consider the
initial value problem
\begin{equation}\label{FPE-init}
  \frac{d}{d t} \langle F, \mu_t\rangle = \langle LF, \mu_t\rangle,
  \quad t>0, \quad \mu_t\bigr|_{t=0}=\mu_0, \quad F\in K\bigl(B_\mathrm{bs}(\Ga_0)\bigr)
\end{equation}
in some space of probability measures on $\bigl(\Ga,\B(\Ga)\bigr)$.
Here the pairing between functions and measure on $\Ga$ is given by
\eqref{pairing}. In fact, the solution to \eqref{FPE-init} describes
the time evolution of distributions instead of the evolution of
initial points in the Markov process. Suppose now that a solution
$\mu_t\in {\mathcal{M}}_{\mathrm{fm} }^1(\Ga )$ to \eqref{FPE-init}
exists and remains locally absolutely continuous with respect to
the Poisson measure $\pi$ for all $t>0$ provided $\mu_0$ has such a property. Then one can consider the correlation functionals
$k_t:=k_{\mu_t}$, $t\geq0$. By \eqref{eqmeans}, we may rewrite
\eqref{FPE-init} in the following way
\begin{equation}\label{ssd0}
  \frac{d}{d t} \langle\!\langle K^{-1}F, k_t\rangle\!\rangle
  = \langle\!\langle K^{-1}LF, k_t\rangle\!\rangle,\quad t>0, \quad
  k_t\bigr|_{t=0}=k_0,
\end{equation}
for all $F\in K\bigl(B_\mathrm{bs}(\Ga_0)\bigr)$. Here the duality
between functions on $\Ga_0$ is given by \eqref{duality} below (cf.
\eqref{duality-intro}). Next, if we substitute $F=KG$, $G\in
B_\mathrm{bs}(\Ga_0)$ in \eqref{ssd0}, we derive
\begin{equation}\label{ssd}
  \frac{d}{d t} \langle\!\langle G, k_t\rangle\!\rangle
  = \langle\!\langle K^{-1}LKG, k_t\rangle\!\rangle, \quad t>0, \quad
  k_t\bigr|_{t=0}=k_0
\end{equation}
for all $G\in B_\mathrm{bs}(\Ga_0)$. In applications, for concrete
birth and death rates we may usually define $(LF)(\eta)$ at least
for all $\eta\in\Ga_0$. In particular, this can be done under
the conditions on birth and death rates described above. Therefore, the
expression $K^{-1}LF$ may be defined via \eqref{k-1trans}
point-wisely. This fact allows us to consider the following operator
\[
(\hat{L}G)(\eta ) :=( K^{-1}LKG)(\eta),\quad \eta\in\Ga_0
\]
for $G\in B_{\mathrm{bs}}\left( \Ga _{0}\right)$. As a result, we
are interested in the weak solution to the equation
\begin{equation}\label{ssd1}
  \frac{\partial}{\partial t} k_t
  = \hat{L}^* k_t, \quad t>0, \quad
  k_t\bigr|_{t=0}=k_0,
\end{equation}
where $L^*$ is dual operator to $\hat{L}$ with respect to the
duality $\langle\!\langle \cdot,\cdot\rangle\!\rangle$. One of the
main aims of the present paper is to study the classical solution to
\eqref{ssd1} in a proper functional space.

To solve \eqref{ssd1}, we will use the following strategy. We start
with a pre-dual (with respect to the duality $\langle\!\langle
\cdot,\cdot\rangle\!\rangle$) initial value problem
\begin{equation}\label{ssd2}
  \frac{\partial}{\partial t} G_t
  = \hat{L} G_t, \quad t>0, \quad
  G_t\bigr|_{t=0}=G_0,
\end{equation}
which will be solved in a Banach space \eqref{space1} of so-called
{\em quasi-observables}. Namely, we construct a holomorphic
semigroup which gives a solution to \eqref{ssd2}. After this we
consider the dual semigroup which produces a weak solution to
\eqref{ssd}. And, finally, we will find a Banach space in which a
classical solution to \eqref{ssd1} exists.

\subsection{Evolutions in the space of quasi-observables}\label{subsect-evol-qo}

We start from the deriving of the expression for $\hat{L}$.
\begin{proposition}\label{prop_desc_oper}
For any $G\in B_{bs}\left( \Ga _{0}\right) $ the following formula
holds \begin{align} (\hat{L}G)(\eta ) =&-\sum_{\xi \subset \eta
}G(\xi )\sum_{x\in \xi }\bigl(K_0^{-1}d(x,\cdot\cup\xi\setminus
x)\bigr)(\eta\setminus
\xi)\notag\\
&+\sum_{\xi \subset \eta }\int_{\R^{d}}\,G(\xi \cup
x)\bigl(K_0^{-1}b(x,\cdot\cup\xi)\bigr)(\eta\setminus \xi) dx,
\qquad \eta\in\Ga_0,\label{newexpr}
\end{align}
provided all terms of the right hand side have sense.
\end{proposition}

\begin{proof}
Following \cite{FKO2009}, for
\begin{equation}\label{Bx-Dx}
B_x=K_0^{-1}b(x,\cdot) , \quad D_x=K_0^{-1}d(x,\cdot)
\end{equation}
we have
\begin{equation}\label{oldexpr}
(\hat{L}G)(\eta )=-\sum_{x\in \eta }\left( D_{x}\star G(\cdot \cup
x)\right) (\eta \setminus x)+\int_{\X}\left( B_{x}\star G(\cdot \cup
x)\right) (\eta )dx.
\end{equation}
Here for the given $\B(\Ga_0)$-measurable functions $G_1$ and $G_2$,
we define
\begin{equation}\label{starconv}
(G_1\star G_2)(\eta )=\sum_{(\eta _1,\eta _2,\eta _3)\in
\mathcal{P}_3(\eta)} G_1(\eta _1\cup\eta _2)G_2(\eta _2\cup\eta _3),
\quad \eta\in\Ga_0,
\end{equation}
where $\mathcal{P}_3(\eta )$ denotes the set of all partitions of
$\eta$ in three parts which may be empty, see~\cite{KK2002}.
Rewriting \eqref{starconv} in the form
\begin{equation}\label{newstarconv}
(G_{1}\star G_{2})(\eta )=\sum_{\xi \subset \eta }G_{1}(\xi )\sum_{\zeta
\subset \xi }G_{2}((\eta \setminus \xi )\cup \zeta ), \quad \eta\in\Ga_0,
\end{equation}
we get
\begin{align*}
(\hat{L}G)(\eta ) =&-\sum_{x\in \eta }\sum_{\xi \subset \eta
\setminus x}G(\xi \cup x)\sum_{\zeta \subset \xi }D_{x}(((\eta
\setminus x)\setminus \xi
)\cup \zeta ) \\
&+\int_{\X}\sum_{\xi \subset \eta }G(\xi \cup x)\sum_{\beta \subset
\xi }B_{x}((\eta \setminus \xi )\cup \beta )dx.
\end{align*}
Using the fact that for any $\B(\Ga_0)$-measurable function $G$
\begin{equation*}
\left( K_{0}G\right) \left( \eta _{1}\cup \eta _{2}\right)
=\sum_{\xi \subset \eta _{1}\cup \eta _{2}}G\left( \xi \right)
=\sum_{\xi _{1}\subset \eta _{1}}\sum_{\xi _{2}\subset \eta
_{2}}G\left( \xi _{1}\cup \xi _{2}\right), \quad
\eta_1\cap\eta_2=\emptyset,
\end{equation*}
for $F=K_{0}G$ we get
\begin{equation}\label{K-1-K}
\bigl(K_{0}^{-1}F\left( \cdot \cup \eta _{2}\right) \bigr)\left( \xi
_{1}\right) =\bigl(K_{0}G\left( \xi _{1}\cup \cdot \right)
\bigr)\left( \eta _{2}\right) , \quad \xi_1\cap\eta_2=\emptyset.
\end{equation}
Now, the simple equality
\begin{equation}\label{changinfoforder}
\sum_{x\in\eta}\sum_{\xi\subset\eta\setminus
x} h(x,\xi,\eta) =\sum_{\xi\subset\eta} \sum_{x\in\xi} h(x,\xi\setminus x ,\eta),
\end{equation}
which holds for any
$\B(\X)\times\B(\Ga_{0})\times\B(\Ga_{0})$-measurable function $h$
finishes the proposition.
\end{proof}

In general, the r.h.s. of \eqref{newexpr} may be undefined. For
arbitrary and fixed $C>1$ we consider the functional space
\begin{equation}
\mathcal{L}_{C}:=L^{1}(\Ga _{0},C^{|\eta |}\la (d\eta )).
\label{space1}
\end{equation}
Throughout of the whole paper, symbol $\left\Vert \cdot \right\Vert
_{C}$ stands for the norm of the space \eqref{space1}. Now we
proceed to study rigorous properties of the operator given by the
expression \eqref{newexpr} in the Banach space $\mathcal{L}_C$.

\begin{remark}
$B_\mathrm{bs}(\Ga_0)$ is a~dense set in $\mathcal{L}_C$.
\end{remark}

\begin{remark}
 The reason to consider the weight $C^{|\cdot|}$ in the definition
 of $\mathcal{L}_C$ is the following. As it was noted above we expect to
 find a solution to \eqref{ssd1} in the space of functions on
 $\Ga_0$ which satisfy the Ruelle bound \eqref{RB-intro}. Such
 space $\K_C$ will be considered in Subsection~\ref{subsect-evol-cf}
 below. The space $\mathcal{L}_C$ is pre-dual to $\K_C$ with respect to
 duality \eqref{duality}.
\end{remark}

Set,
\begin{align}
D\left( \eta \right) :=&\sum_{x\in \eta }d\left( x,\eta
\setminus x\right) \geq 0,\quad\eta \in \Ga _{0}; \label{mult}\\
\D :=&\left\{ G\in \mathcal{L}_{C}~|~D\left( \cdot \right) G\in
\mathcal{L}_{C}\right\} .\label{dom}
\end{align}
Note that $B_\mathrm{bs}(\Ga_0) \subset \D$. In particular, $\D$ is
a dense set in $\mathcal{L}_C$.

We will show that $(\hat{L},\D)$ given by \eqref{newexpr},
\eqref{dom} generates $C_0$-semigroup on $\mathcal{L}_C$.

\begin{theorem}\label{th1}
Suppose that there exists $a_{1}\geq1$, $a_2>0$ such that for all
$\xi \in \Ga _{0}$ and a.a. $x\in\X$
\begin{align}
\sum_{x\in\xi}\int_{\Ga _{0}}\left\vert K_{0}^{-1}d\left( x,\cdot
\cup \xi \setminus x\right) \right\vert \left( \eta \right)
C^{\left\vert \eta \right\vert }d\la \left( \eta \right) \leq &
a_{1} D(\xi) ,
\label{est-d} \\
\sum_{x\in\xi}\int_{\Ga _{0}}\left\vert K_{0}^{-1}b\left( x,\cdot
\cup \xi \setminus x\right) \right\vert \left( \eta \right)
C^{\left\vert \eta \right\vert }d\la \left( \eta \right) \leq
&a_{2}D(\xi) . \label{est-b}
\end{align}
and, moreover,
\begin{equation}\label{asmall}
a_1+\frac{a_2}{C}<\frac{3}{2}.
\end{equation}
Then $(\hat{L},\D)$ is the generator of a~holomorphic semigroup
$\hat{T}\left( t\right) $ on $\mathcal{L}_{C}$.
\end{theorem}

\begin{remark}
 Conditions \eqref{est-d}--\eqref{asmall} express an essential
 role of the death rate in our construction. They are crucial
 for the existence of the classical solution
 to the evolution equation \eqref{ssd2} in the space $\mathcal{L}_C$ of
 quasi-observables (cf.~Remark~\ref{necess} below). Note also,
 that alternatively to a semigroup approach one can
 study local in time solutions to \eqref{ssd2} also.
 For a particular model it was realized in the recent paper \cite{FKKoz2011}.
\end{remark}

\begin{proof}
Let us consider the multiplication operator $\left( L_{0},\D\right)
$ on $ \mathcal{L}_{C}$ given by
\begin{equation} (L_{0}G)(\eta )=-D\left( \eta \right) G(\eta
),\quad G\in \D,\ \ \eta \in \Ga _{0}. \label{L0oper}
\end{equation}
We recall that a~densely defined closed operators $A$ on
$\mathcal{L}_{C}$ is called sectorial of angle $\omega\in
(0,\,\frac{ \pi }{2})$ if its resolvent set $\rho (A)$ contains the
sector
\begin{equation*}
\mathrm{Sect}\left( \frac{\pi }{2}+\omega \right) :=\left\{ z \in
\mathbb{C}\,\Bigm||\arg z |<\frac{\pi }{2}+\omega \right\} \setminus\{0\}
\end{equation*} and for each $\eps\in(0;\omega)$ there exists
$M_\eps\geq 1$ such that
\begin{equation}\label{resbound}
||R(z,A)||\leq \frac{M_{\eps }}{|z|}
\end{equation}
for all $z\neq 0$ with $|\arg z |\leq \dfrac{\pi }{2}+\omega -\eps.$
Here and below we will use notation
\[
R(z,A):=(z\1 -A)^{-1}, \quad z\in\rho(A).
\]
The set of all sectorial operators of angle $\omega\in (0,\,\frac{
\pi }{2})$ in $\mathcal{L}_C$ we denote by $\mathcal{H} _{C}(\omega
)$. Any $A\in\mathcal{H} _{C}(\omega )$ is a~generator of a~bounded
semigroup $T(t)$ which is holomorphic in the sector $|\arg
\,t|<\omega $ (see e.g. \cite[Theorem II.4.6]{EN2000}). One can
prove the following lemma.
\begin{lemma}\label{lem1}
The operator $\left( L_{0},\D\right) $ given by \eqref{L0oper} is a
generator of a~contraction semigroup on $\mathcal{L}_{C}.$ Moreover,
$L_{0}\in \mathcal{H}_{C}(\omega)$ for all $\omega \in (0,\,\frac{
\pi }{2})$ and \eqref{resbound} holds with
$M_\eps=\frac{1}{\cos\omega}$ for all $\eps\in(0;\omega)$.
\end{lemma}
\begin{proof}[Proof of Lemma~\ref{lem1}]
It is not difficult to show that the densely defined operator
$L_{0}$ is closed in~$\mathcal{L}_{C}$. Let $0<\omega <\frac{\pi
}{2}$ be arbitrary and fixed. Clear, that for all $z \in
\mathrm{Sect}\left( \frac{\pi }{2} +\omega \right) $
\begin{equation*}
\bigl|D\left( \eta \right) +z \bigr|>0,\quad \eta \in \Ga _{0}.
\end{equation*} Therefore, for any $z \in \mathrm{Sect}\left( \frac{\pi }{2}+\omega
\right) $ the inverse operator $R(z,L_{0})=(z\1
- L_{0})^{-1}$, the action of
which is given by
\begin{equation}
\lbrack R(z,L_{0})G](\eta )=\frac{1}{D\left( \eta \right) +z
}\,G(\eta ), \label{necf}
\end{equation} is well defined on the whole space $\mathcal{L}_{C}$. Moreover,
\[
|D(\eta)+z|=\sqrt{(D(\eta)+\Re z)^2+(\Im
z)^2}\geq
\begin{cases}
|z|, &\mathrm{if} \ \Re z\geq 0\\
|\Im\,z|, &\mathrm{if} \ \Re z<0
\end{cases},
\]
and for any $z\in\Sect\left(\frac{\pi}{2}+\omega\right)$
\[
|\Im\,z|=|z| |\sin \arg z|\geq|z|\left|\sin\left(\frac{\pi}{2}+\omega\right)\right|=|z|\cos\omega.
\]
As a~result,
 for any $z\in\Sect\left(\frac{\pi}{2}+\omega\right)$\begin{equation}\label{resbound-ex}
||R(z,L_{0})||\leq \frac{1}{|z |\cos\omega},
\end{equation}
that implies the second assertion. Note also that
$|D(\eta)+z|\geq\Re z$ for $\Re z>0$, hence,
\begin{equation}\label{HY}
||R(z,L_{0})||\leq
\frac{1}{\Re z},
\end{equation}
that proves the first statement by the classical
Hille--Yosida theorem.
\end{proof}
\noindent For any $G\in B_{bs}\left( \Ga _{0}\right) $ we define
\begin{align}
\nonumber \left( L_{1}G\right) \left( \eta \right)
:=&\,(\hat{L}G)(\eta)-(L_0G)(\eta)\\=&-\sum_{\xi\subsetneq \eta
}G(\xi )\sum_{x\in \xi }\bigl(K_0^{-1}d(x,\cdot\cup\xi\setminus
x)\bigr)(\eta\setminus
\xi)\notag\\
&+\sum_{\xi \subset \eta }\int_{\R^{d}}\,G(\xi \cup
x)\bigl(K_0^{-1}b(x,\cdot\cup\xi)\bigr)(\eta\setminus \xi)
dx.\label{operL1def}
\end{align}
Next Lemma shows that, under conditions \eqref{est-d}, \eqref{est-b}
above, the operator $L_1$ is relatively bounded by the operator
$L_0$.
\begin{lemma}\label{lem2}
Let \eqref{est-d}, \eqref{est-b} hold. Then $(L_1,\D)$ is a
well-defined operator in $\mathcal{L}_C$ such that
\begin{equation}
\left\Vert L_{1}R(z,L_{0})\right\Vert \leq a_1-1+\frac{a_2}{C}, \quad \Re z>0\label{relbound}
\end{equation}
and
\begin{equation}\label{deep}
\|L_1 G\|\leq \Bigl(a_1-1+\frac{a_2}{C}\bigr)\|L_0G\|,
\quad G\in\D.
\end{equation}
\end{lemma}
\begin{proof}[Proof of Lemma~\ref{lem2}]
By Lemma~\ref{Minlos}, we have for any $G\in \mathcal{L}_{C}$, $\Re
z>0$
\begin{align*}
&\int_{\Ga _{0}}\biggl\vert -\sum_{\xi\subsetneq \eta
}\frac{1}{z+D(\xi)}G(\xi )\sum_{x\in \xi
}\bigl(K_0^{-1}d(x,\cdot\cup\xi\setminus x)\bigr)(\eta\setminus
\xi)\biggr\vert C^{\left\vert \eta \right\vert }d\la
\left( \eta \right) \\
\leq &\int_{\Ga _{0}}\sum_{\xi \subsetneq \eta
}\frac{1}{|z+D(\xi)|}\left\vert G(\xi )\right\vert \sum_{x\in \xi
}\bigl|K_0^{-1}d(x,\cdot\cup\xi\setminus x)\bigr|(\eta\setminus \xi)
C^{\left\vert \eta \right\vert
}d\la \left( \eta \right) \\
=&\int_{\Ga _{0}}\frac{1}{|z+D(\xi)|}\left\vert G(\xi )\right\vert
\sum_{x\in \xi }\int_{\Ga
_{0}}\bigl|K_0^{-1}d(x,\cdot\cup\xi\setminus
x)\bigr|(\eta)C^{\left\vert \eta \right\vert }d\la \left( \eta
\right) C^{\left\vert \xi \right\vert
}d\la \left( \xi \right) \\
&-\int_{\Ga _{0}}\frac{1}{|z+D(\eta)|}D\left( \eta \right)
\left\vert G(\eta )\right\vert C^{\left\vert \eta \right\vert }d\la
\left( \eta
\right)\\
\leq&(a_{1}-1)\int_{\Ga _{0}} \frac{1}{\Re z+D(\eta)}
D(\eta)|G(\eta)|C^{|\eta|} d\la(\eta)\leq(a_1-1)\|G\|_C,
\end{align*} and \begin{align*}
&\int_{\Ga _{0}}\biggl\vert \sum_{\xi \subset \eta }\int_{\R
^{d}}\,\frac{1}{z+D(\xi\cup x)}G(\xi \cup
x)\bigl(K_0^{-1}b(x,\cdot\cup\xi)\bigr) (\eta\setminus
\xi)dx\biggr\vert C^{\left\vert \eta \right\vert }d\la
\left( \eta \right) \\
\leq &\int_{\Ga _{0}}\int_{\Ga
_{0}}\int_{\R^{d}}\,\frac{1}{|z+D(\xi\cup x)|}\left\vert G(\xi \cup
x)\right\vert \left\vert K_0^{-1}b(x,\cdot\cup\xi) \right\vert
(\eta)dxC^{\left\vert \eta \right\vert }C^{\left\vert \xi
\right\vert }d\la \left( \xi \right) d\la \left(
\eta \right) \\
=&\frac{1}{C}\int_{\Ga _{0}}\,\frac{1}{|z+D(\xi)|}\left\vert G(\xi
)\right\vert \sum_{x\in \xi }\int_{\Ga _{0}} \left\vert
K_0^{-1}b(x,\cdot\cup\xi\setminus x) \right\vert (\eta)
C^{\left\vert \eta \right\vert }d\la \left( \eta \right)
C^{\left\vert \xi \right\vert }d\la \left( \xi \right) \\
\leq&\frac{a_{2}}{C}\int_{\Ga _{0}}\,\frac{1}{\Re
z+D(\xi)}\left\vert G(\xi )\right| D(\xi) C^{\left\vert \xi
\right\vert }d\la \left( \xi \right)\leq\frac{a_2}{C}\|G\|_C.
\end{align*}
Combining these inequalities we obtain \eqref{relbound}.
The same considerations yield
\begin{align*}
&\int_{\Ga _{0}}\biggl\vert -\sum_{\xi\subsetneq \eta }G(\xi
)\sum_{x\in \xi }\bigl(K_0^{-1}d(x,\cdot\cup\xi\setminus
x)\bigr)(\eta\setminus \xi)\biggr\vert C^{\left\vert \eta
\right\vert }d\la
\left( \eta \right) \\
&\quad+\int_{\Ga _{0}}\biggl\vert \sum_{\xi \subset \eta }\int_{\R
^{d}}\,G(\xi \cup x)\bigl(K_0^{-1}b(x,\cdot\cup\xi)\bigr)
(\eta\setminus \xi)dx\biggr\vert C^{\left\vert \eta \right\vert
}d\la
\left( \eta \right) \\
\leq&\left((a_{1}-1)+\frac{a_{2}}{C}\right)\int_{\Ga
_{0}}\,\left\vert G(\eta )\right| D(\eta) C^{\left\vert \eta
\right\vert }d\la \left( \eta \right),
\end{align*}
that proves \eqref{deep} as well.
\end{proof}
\noindent And now we proceed to finish the proof of the
Theorem~\ref{th1}. Let us set
$\theta:=a_1+\frac{a_2}{C}-1\in\bigl(0;\frac{1}{2}\bigr)$. Then
$\frac{\theta}{1-\theta}\in(0;1)$. Let $\omega\in
\bigl(0;\frac{\pi}{2}\bigr)$ be such that
$\cos\omega<\frac{\theta}{1-\theta}$. Then, by the proof of
Lemma~\ref{lem1}, $L_{0}\in \mathcal{H}_{C}(\omega)$ and
$||R(z,L_0)||\leq \frac{M}{|z|}$ for all $z\neq 0$ with $|\arg z
|\leq \dfrac{\pi }{2}+\omega $, where $M:=\frac{1}{\cos \omega}$.
Then
$$\theta=\frac{1}{1+\frac{1-\theta}{\theta}}<\frac{1}{1+\frac{1}{\cos\omega}}=\frac{1}{1+M}.$$
Hence, by \eqref{deep} and the proof of \cite[Theorem
III.2.10]{EN2000}, we have that $(\hat{L}=L_0+L_1,\D)$ is a
generator of holomorphic semigroup on $\mathcal{L}_C$.
\end{proof}

\begin{remark}
By \eqref{mult}, the estimates \eqref{est-d}, \eqref{est-b} are
satisfied if
\begin{align}
\int_{\Ga _{0}}\left\vert K_{0}^{-1}d\left( x,\cdot \cup \xi \right)
\right\vert \left( \eta \right) C^{\left\vert \eta \right\vert }d\la
\left( \eta \right) \leq & a_{1} d\left( x,\xi \right) , \label{est-d1} \\
\int_{\Ga _{0}}\left\vert K_{0}^{-1}b\left( x,\cdot \cup \xi \right)
\right\vert \left( \eta \right) C^{\left\vert \eta \right\vert }d\la
\left( \eta \right) \leq &a_{2}d\left( x,\xi \right) .
\label{est-b1}
\end{align}
\end{remark}

\begin{example} (Glauber-type dynamics in continuum). \label{ex-Gl}
Let $L$ be given by \eqref{BaDGen} with
\begin{align}
d(x,\ga\setminus x)&=\exp\Bigl\{s\sum_{y\in\ga\setminus x}\phi(x-y)\Bigr\}, &&x\in\ga,\ \ga\in\Ga,\label{Gl-d}\\
b(x,\ga)&=z\exp\Bigl\{(s-1)\sum_{y\in\ga}\phi(x-y)\Bigr\}, &&x\in\X\setminus\ga,\ \ga\in\Ga,\label{Gl-b}
\end{align}
where $\phi:\X\rightarrow\R_+$ is a~pair potential,
$\phi(-x)=\phi(x)$, $z>0$ is an activity parameter and $s\in[0;1]$.
For any $s\in[0;1]$ the operator $L$ is well defined and, moreover,
symmetric in the space $L^2(\Ga,\mu)$, where $\mu$ is a~Gibbs
measure, given by the pair potential $\phi$ and activity parameter
$z$ (see e.g. \cite{KLR2007} and references therein). This gives
possibility to study the corresponding semigroup in $L^2(\Ga,\mu)$.
In the case $s=0$, the corresponding dynamics was also studied in
another Banach spaces, see e.g. \cite{KKZ2006, FKKZ2010, FKK2010}.
Below we show that one of the main result of the paper stated in
Theorem~\ref{th1} can be applied to the case of arbitrary
$s\in[0;1]$. Set
\begin{equation}\label{integr_cond}
\beta_{\tau}:=\int_\X\bigl\vert e^{\tau\phi(x)}-1\bigr\vert
dx\in[0;\infty], \quad \tau\in[-1;1].
\end{equation}
Let $s$ be arbitrary and fixed. Suppose that $\beta_s<\infty$,
$\beta_{s-1}<\infty$. Then,
 by \eqref{Gl-d}, \eqref{Kexp}, and \eqref{intexp}
\begin{align*}
 K_{0}^{-1}d\left( x,\cdot \cup \xi \right)
 \left( \eta \right)&=d(x,\xi)\,e_\la(e^{s\phi(x-\cdot)}-1,\eta),\\
\int_{\Ga _{0}}\left\vert K_{0}^{-1}d\left( x,\cdot \cup \xi \right)
\right\vert \left( \eta \right) C^{\left\vert \eta \right\vert }d\la
\left( \eta \right)&=d(x,\xi)e^{C\beta_{s}},
\end{align*}
and, analogously,
\[
\int_{\Ga _{0}}\left\vert K_{0}^{-1}b\left( x,\cdot \cup \xi \right)
\right\vert \left( \eta \right) C^{\left\vert \eta \right\vert }d\la
\left( \eta \right)=b(x,\xi)e^{C\beta_{s-1}}\leq
zd(x,\xi)e^{C\beta_{s-1}},
\]
since $\phi\geq0$. Therefore, to apply Theorem~\ref{th1} we should
assume additionally that
\begin{equation}\label{sn0smallz}
e^{C\beta_{s}} + \frac{z}{C}e^{C\beta_{s-1}}<\frac{3}{2}.
\end{equation}
In particular, for $s=0$ we obtain the condition
(cf. \cite{KKZ2006})
\begin{equation}\label{s0smallz}
\frac{z}{C}e^{C\beta_{-1}}<\frac{1}{2}.
\end{equation}
\end{example}

\begin{example}\label{ex-BDLP}
(Bolker--Dieckman--Law--Pacala (BDLP) model) This example describes
the model of plant ecology, see \cite{FKK2009} and references
therein. Let $L$ be given by \eqref{BaDGen} with
\begin{align}
d(x,\ga\setminus x)&=m + \varkappa^-\sum_{y\in\ga\setminus x}a^-(x-y), &&x\in\ga,\ \ga\in\Ga,\label{BDLP-d}\\
b(x,\ga)&=\varkappa^+\sum_{y\in\ga}a^+(x-y), &&x\in\X\setminus\ga,\
\ga\in\Ga,\label{BDLP-b}
\end{align}
where $m>0$, $\varkappa^\pm\geq0$, $0\leq a^\pm\in L^1(\X,dx)\cap
L^\infty(\X,dx)$, $\int_\X a^\pm(x)dx=1$. Then
\begin{gather*}
K_{0}^{-1}d\left( x,\cdot \cup \xi \right)
 \left( \eta \right)=d(x,\xi)0^{|\eta|}+\varkappa^-\1_{\Ga^{(1)}}(\eta)\sum_{y\in\eta}a^{-}(x-y),\\
\int_{\Ga _{0}}\left\vert K_{0}^{-1}d\left( x,\cdot \cup \xi \right)
\right\vert \left( \eta \right) C^{\left\vert \eta \right\vert }d\la
\left( \eta \right)=d(x,\xi)+C\varkappa^-,
\end{gather*}
and, analogously,
\[
\int_{\Ga _{0}}\left\vert K_{0}^{-1}b\left( x,\cdot \cup \xi \right)
\right\vert \left( \eta \right) C^{\left\vert \eta \right\vert }d\la
\left( \eta \right)=b(x,\xi)+C\varkappa^+.
\]
Therefore, if we suppose, for example, that (cf. \cite{FKK2009})
\begin{align}
 4\varkappa^-C&<m\label{smallparBDLP-1}\\
4\varkappa^+a^+(x)&\leq C\varkappa^-a^-(x),
\quad x\in\X,\label{smallparBDLP-2}
\end{align}
then there exists $\delta>0$ such that
\[
d(x,\xi)+C\varkappa^-\leq d(x,\xi)+\frac{m}{4+\delta}\leq\Bigl(1+\frac{1}{4+\delta}\Bigr)d(x,\xi)
\]
and
\[
b(x,\xi)+C\varkappa^+\leq\frac{C}{4}\varkappa^-\sum_{y\in\xi}a^-(x-y)+\frac{Cm}{16}<\frac{C}{4}d(x,\xi),
\] since $4\varkappa^+\leq
C\varkappa^-<\frac{m}{4}$. The last bound we get integrating both
sides of \eqref{smallparBDLP-2} over $\X$.

Hence, \eqref{est-d}, \eqref{est-b} hold
and
\[
a_1+\frac{a_2}{C}=1+\frac{1}{4+\delta}+\frac{1}{4}<\frac{3}{2}.
\]
\end{example}

\begin{remark}\label{necess}
 It was shown in \cite{FKK2009} that the condition
 like \eqref{smallparBDLP-1} is essential. Namely,
 if $m>0$ is arbitrary small the operator
 $\hat{L}$ will not be even accretive in $\mathcal{L}_C$.
\end{remark}

\subsection{Evolutions in the space of correlation
functions}\label{subsect-evol-cf}

In this Subsection we will use the semigroup $\hat{T}(t)$ acting oh
the space of quasi-observables for a construction of solution to the
evolution equation \eqref{ssd1} on space of correlation functions.

We denote $d\la _{C}:=C^{|\cdot |}d\la $; and the dual space $(
\mathcal{L}_{C})^{\prime }=\bigl(L^{1}(\Ga _{0},d\la _{C})\bigr)
^{\prime }=L^{\infty }(\Ga _{0},d\la _{C})$. The space $(\mathcal{L}
_{C})^{\prime }$ is isometrically isomorphic to the Banach space
\begin{equation*}
{\K}_{C}:=\left\{ k:\Ga _{0}\rightarrow {\R}\,\Bigm| k\cdot
C^{-|\cdot |}\in L^{\infty }(\Ga _{0},\la )\right\}
\end{equation*} with the norm
\[
\Vert k\Vert _{{\K}_{C}}:=\Vert C^{-|\cdot |}k(\cdot )\Vert
_{L^{\infty }(\Ga _{0},\la )},
\]
where the isomorphism is given by the isometry $R_{C}$
\begin{equation}
(\mathcal{L}_{C})^{\prime }\ni k\longmapsto R_{C}k:=k\cdot C^{|\cdot
|}\in { \K}_{C}. \label{isometry}
\end{equation}

In fact, one may consider the duality between the Banach spaces
$\mathcal{L} _{C}$ and ${\K}_{C}$ given by the following expression
\begin{equation}
\left\langle \!\left\langle G,\,k\right\rangle \!\right\rangle
:=\int_{\Ga _{0}}G\cdot k\,d\la ,\quad G\in \mathcal{L}_{C},\ k\in {
\K}_{C} \label{duality}
\end{equation} with $\left\vert \left\langle \!\left\langle G,k\right\rangle
\!\right\rangle \right\vert \leq \Vert G\Vert _{C}\cdot \Vert k\Vert
_{{ \K}_{C}}$. It is clear that $k\in {\K}_{C}$ implies
\[
|k(\eta )|\leq \Vert k\Vert _{{\K}_{C}}\,C^{|\eta|} \qquad
\mathrm{for} \ \la\text{-a.a.} \ \eta \in \Ga _{0}.
\]

Let $\bigl({\hat{L}}^{\prime },\Dom({\hat{L}}^{\prime })\bigr)$ be
an operator in $(\mathcal{L}_{C})^{\prime }$ which is dual to the
closed operator $\bigl( {\hat{L}},\D\bigr)$. We consider also its
image on ${\K} _{C}$ under the isometry $R_{C}$. Namely, let
${\hat{L}}^{\ast }=R_{C}{\hat{L }}^{\prime }R_{C^{-1}}$ with the
domain $\Dom({\hat{L}}^{\ast })=R_{C}\Dom({\hat{L}}^{\prime })$.

Similarly, one can consider the adjoint semigroup $\hat{T}^{\prime
}(t)$ in $(\mathcal{L}_{C})^{\prime }$ and its image $\hat{T}^{\ast
}(t)$ in ${\K}_{C}$. The space $\mathcal{L}_C$ is not reflexive,
hence, $\hat{T}^{\ast }(t)$ is not $C_0$-semigroup in ${\K}_{C}$.
However, from the general theory (see e.g. \cite{EN2000}) the last
semigroup will be weak*-continuous, weak*-differentiable at $0$ and
${\hat{L}}^\ast$ will be weak*-generator of ${\hat{T}}^\ast(t)$.
Therefore, one has an evolution in the space of correlation
functions. In fact, we have a~solution to the evolution equation
\eqref{ssd1}, in a~weak*-sense. This subsection is devoted to the
study of a~strong solution to this equation.

\begin{proposition}\label{pr3}
 Let \eqref{est-d}, \eqref{est-b} be satisfied. Suppose that there exists
$A>0$, $N\in \N_{0}$, $\nu\geq 1$ such that for $\xi \in \Ga _{0}$
and $x\notin \xi $
\begin{equation}
d\left( x,\xi \right) \leq
A(1+\left\vert \xi \right\vert) ^{N}\nu^{\left\vert \xi \right\vert
}, \qquad \label{D-bdd}
\end{equation} Then for any $\alpha \in \left( 0;\frac{1}{\nu}\right) $
\begin{equation}\label{incldom}
\K_{\alpha C}\subset \Dom({\hat{L}}^{\ast }).
\end{equation}
\end{proposition}

\begin{proof}
In order to show \eqref{incldom} it is enough to verify that for any
$k\in\K_{\alpha C}$ there exists $k^*\in\K_C$ such that for any
$G\in\Dom(\hat{L})$
\begin{equation}\label{pairdualneed}
\bigl\langle \!\bigl\langle \hat{L} G,\,k\bigr\rangle \!\bigr\rangle
=\left\langle \!\left\langle G,\,k^*\right\rangle \!\right\rangle.
\end{equation}

According to \cite{FKO2009}, \eqref{pairdualneed} is valid for any
$k\in\K_{\alpha C}$ with $k^*=\hat{L}^{\ast }k$, where
\begin{align*}
(\hat{L}^{\ast }k)(\eta )=&-\int_{\Ga _{0}} k(\zeta
\cup \eta )\sum_{x\in \eta }\sum_{\xi \subset \eta \setminus
x}D_{x}(\zeta
\cup \xi )d\la (\zeta ) \\
& +\int_{\Ga _{0}}\sum_{x\in \eta }k(\zeta \cup (\eta \setminus
x))\sum_{\xi \subset \eta \setminus x}B_{x}(\zeta \cup \xi )d\la
(\zeta ),
\end{align*}
provided $k^*\in\K_C$. Using \eqref{K-1-K}, one can rewrite the last
expression
\begin{align*}
(\hat{L}^{\ast }k)(\eta )=&-\sum_{x\in \eta }\int_{\Ga _{0}} k(\zeta
\cup \eta )\bigl( K_0^{-1}d(x,\cdot\cup \eta\setminus x)\bigr)
(\zeta)d\la (\zeta ) \\
& +\sum_{x\in \eta }\int_{\Ga _{0}}k(\zeta \cup (\eta \setminus
x))\bigl( K_0^{-1}b(x,\cdot\cup \eta\setminus x)\bigr) (\zeta)d\la
(\zeta).
\end{align*}
Then, by \eqref{est-d}, \eqref{est-b}, and \eqref{D-bdd},
\begin{align*}
&C^{-\left\vert \eta \right\vert }\left\vert (\hat{L}^{\ast }k)(\eta
)\right\vert \\
\leq &C^{-\left\vert \eta \right\vert }\sum_{x\in \eta }\int_{\Ga
_{0}}\!\left\vert k(\zeta \cup \eta )\right\vert \bigl|
K_0^{-1}d(x,\cdot\cup \eta\setminus x)\bigr|
(\zeta) d\la (\zeta ) \\
&+C^{-\left\vert \eta \right\vert }\sum_{x\in \eta }\int_{\Ga
_{0}}\!\,\left\vert k(\zeta \cup (\eta \setminus x))\right\vert
\bigl| K_0^{-1}b(x,\cdot\cup \eta\setminus x)\bigr| (\zeta)d\la
(\zeta
)\, \\
\leq &\left\Vert k\right\Vert _{\K_{\alpha C}}\alpha ^{\left\vert
\eta \right\vert }\sum_{x\in \eta }\int_{\Ga _{0}}\!\left( \alpha
C\right) ^{\left\vert \zeta \right\vert} \bigl|
K_0^{-1}d(x,\cdot\cup \eta\setminus x)\bigr|
(\zeta)d\la (\zeta ) \\
&+\frac{1}{\alpha C}\left\Vert k\right\Vert _{\K_{\alpha C}}\alpha
^{\left\vert \eta \right\vert }\sum_{x\in \eta }\int_{\Ga
_{0}}\!\,\left( \alpha C\right) ^{\left\vert \zeta \right\vert }
\bigl| K_0^{-1}b(x,\cdot\cup \eta\setminus x)\bigr|
(\zeta) d\la (\zeta ) \\
\leq &\,\left\Vert k\right\Vert _{\K_{\alpha C}}\left( a_1
+\frac{a_2}{\alpha C}\right)\alpha ^{\left\vert \eta \right\vert
}\sum_{x\in \eta
} d\left( x,\eta \setminus x\right) \\
\leq &\,A\left\Vert k\right\Vert _{\K_{\alpha C}}\left( a_1
+\frac{a_2}{\alpha C}\right)\alpha ^{\left\vert \eta \right\vert
}(1+\left\vert \eta \right\vert) ^{N+1}\nu^{\left\vert \eta
\right\vert -1}.
\end{align*} Using elementary inequality \begin{equation}
(1+t)^{b}a^{t}\leq \frac{1}{a}\left( \frac{b}{-e\ln a}\right)
^{b},~~b\geq 1,~a\in \left( 0;1\right) ,~t\geq 0, \label{bdd}
\end{equation} we have for $\alpha \nu <1$ \begin{equation*}
\esssup_{\eta \in \Ga _{0}}C^{-\left\vert \eta \right\vert
}\left\vert (\hat{L}^{\ast }k)(\eta )\right\vert \leq \left\Vert
k\right\Vert _{\K_{\alpha C}}\left( a_1 +\frac{a_2}{\alpha
C}\right)\frac{A }{ \alpha \nu^2}\left( \frac{N+1}{-e\ln \left(
\alpha \nu\right) }\right) ^{N+1}<\infty. \qedhere
\end{equation*}
\end{proof}

\begin{lemma}
Let \eqref{D-bdd} holds. We define
for any $\alpha\in (0;1)$
\begin{equation*}
\D_{\alpha }:\mathcal{=}\left\{ G\in \mathcal{L}_{\alpha
C}~|~D\left( \cdot \right) G\in \mathcal{L}_{\alpha C}\right\} .
\end{equation*} Then for any $\alpha\in (0;\frac{1}{\nu})$ \begin{equation}
\D\subset \mathcal{L}_{C}\subset \D_{\alpha }\subset
\mathcal{L}_{\alpha C} \label{subs}
\end{equation}
\end{lemma}
\begin{proof}
The first and last inclusions are obvious. To prove the second one,
we use \eqref{D-bdd}, \eqref{bdd} and obtain for any $G\in
\mathcal{L}_{C}$
\begin{align*}
\int_{\Ga _{0}}D\left( \eta \right)
\left\vert G\left( \eta \right) \right\vert \left( \alpha C\right)
^{\left\vert \eta \right\vert }d\la \left( \eta \right) \leq
&\int_{\Ga _{0}}\alpha ^{\left\vert \eta \right\vert }\sum_{x\in
\eta }A(1+\left\vert \eta \right\vert) ^{N}\nu^{\left\vert \eta
\right\vert -1}\left\vert G\left( \eta \right) \right\vert
C^{\left\vert \eta \right\vert }d\la \left( \eta \right)
\\\leq &\,\mathrm{const}\int_{\Ga _{0}}\left\vert G\left(
\eta \right) \right\vert C^{\left\vert \eta \right\vert }d\la \left(
\eta \right) <\infty. \qedhere
\end{align*}
\end{proof}

\begin{proposition}
Let \eqref{est-d}, \eqref{est-b}, and \eqref{D-bdd} hold with
\begin{equation}\label{asmallalpha}
a_1+\frac{a_2}{\alpha C}<\frac{3}{2}
\end{equation}
for some $\alpha\in (0;1)$. Then $(\hat{L},\D_\alpha)$ is
a~generator of a~holomorphic semigroup $\hat{T}_\alpha\left(
t\right) $ on $\mathcal{L}_{\alpha C}$.
\end{proposition}

\begin{proof}
The proof is similar to the proof of Theorem~\ref{th1}, taking into
account that bounds \eqref{est-b}, \eqref{est-d} imply the same
bounds for $\alpha C$ instead of $C$. Note also that
\eqref{asmallalpha} is stronger than \eqref{asmall}.
\end{proof}

Under conditions \eqref{est-d}, \eqref{est-b}, \eqref{asmallalpha},
and \eqref{D-bdd}, we consider the adjoint semigroup
$\hat{T}^{\prime }(t)$ in $(\mathcal{L}_{C})^{\prime }$ and its
image $\hat{T}^{\ast }(t)$ in ${\K}_{C}$.
By~e.g.~\cite[Subsection~II.2.6]{EN2000}, the restriction
$\hat{T}^{\odot }(t)$ of the semigroup $\hat{T}^{\ast }(t)$ onto its
invariant Banach subspace $\overline{ \Dom({\hat{L}}^{\ast })}$
(here and below all closures are in the norm of the space
${\K}_{C}$) is a~strongly continuous semigroup. Moreover, its
generator ${\hat{L}} ^\odot$ will be a~part of ${\hat{L}}^\ast$,
namely,
\[
\Dom({\hat{L}} ^\odot)=\Bigl\{k\in \Dom({\hat{L}}^\ast) \Bigm|
{\hat{L}}^\ast k\in \overline{\Dom({\hat{L}}^\ast)}\Bigr\}
\]
and ${\hat{L} }^\ast k ={\hat{L}}^\odot k$ for any $k\in
\Dom({\hat{L}}^\odot)$.

\begin{theorem}\label{hint}
Let \eqref{est-d}, \eqref{est-b}, and \eqref{D-bdd} hold with
\begin{equation}\label{nusmall}
1\leq\nu<\frac{C}{a_2}\biggl(\frac{3}{2}-a_1\biggr).
\end{equation}
Then for any $\alpha \in \left(
\dfrac{a_{2}}{C\bigl(\frac{3}{2}-a_{1}\bigr)};\dfrac{1}{\nu}\right)
$ the set $\overline{ \K_{\alpha C}}$ is a~$\hat{T}^{\odot
}(t)$-invariant Banach subspace of $\K_{C}$.
\end{theorem}

\begin{proof} First of all note that the
condition on $\alpha$ implies \eqref{asmallalpha}. Next, we prove
that $\hat{T}_{\alpha }\left( t\right) G=\hat{T} \left( t\right) G$
for any $G\in \mathcal{L}_{C}\subset \mathcal{L}_{\alpha C}$. Let
$\hat{L}_\alpha=(\hat{L},\D_{\alpha })$ is the operator in
$\mathcal{L}_{\alpha C}$. There exists $\omega>0$ such that
$(\omega;+\infty)\subset\rho(\hat{L})\cap\rho(\hat{L}_\alpha)$, see
e.g. \cite[Section III.2]{EN2000}. For some fixed $z
\in(\omega;+\infty)$ we denote by $R(z,\hat{L})=\bigl( z\1-
\hat{L}\bigr) ^{-1}$ the resolvent of $(\hat{L},\D)$ in
$\mathcal{L}_{C}$ and by $R(z,\hat{L}_\alpha)=\bigl(
z\1-\hat{L}_\alpha \bigr) ^{-1} $ the resolvent of $\hat{L}_\alpha$
in $\mathcal{L}_{\alpha C}$. Then for any $G\in \mathcal{L}_{C}$ we
have $R(z,\hat{L})G\in \D\subset \D_{\alpha }$ and \begin{equation*}
R(z,\hat{L})G-R(z,\hat{L}_\alpha)G=R(z,\hat{L}_\alpha)\bigl( (z
\1-\hat{L}_\alpha) -\bigl( z\1-\hat{L}) \bigr)R(z,\hat{L})G=0,
\end{equation*}
since $\hat{L}_\alpha=\hat{L}$ on $\D$. As a~result,
$\hat{T}_{\alpha }\left( t\right) G=\hat{T}\left( t\right) G$ on
$\mathcal{L}_C$.

Note that for any $G\in \mathcal{L}_{C}\subset \mathcal{L}_{\alpha
C}$ and for any $k\in {\K}_{\alpha C}\subset {\K}_{C}$ we have $
\hat{T}_{\alpha }(t)G\in \mathcal{L}_{\alpha C}$ and
\begin{equation*}
\left\langle \!\!\left\langle \hat{T}_{\alpha }(t)G,k\right\rangle
\!\!\right\rangle =\left\langle \!\!\left\langle G,\hat{T}_{\alpha
}^{\ast }(t)k\right\rangle \!\!\right\rangle ,
\end{equation*} where, by the same construction as before, $\hat{T}_{\alpha }^{\ast
}(t)k\in {\K}_{\alpha C}$. But $G\in \mathcal{L}_{C}$, $k\in {\K}
_{C}$ implies
\begin{equation*}
\left\langle \!\!\left\langle \hat{T}_{\alpha }(t)G,k\right\rangle
\!\!\right\rangle =\left\langle \!\!\left\langle
\hat{T}(t)G,k\right\rangle \!\!\right\rangle =\left\langle
\!\!\left\langle G,\hat{T}^{\ast }(t)k\right\rangle
\!\!\right\rangle .
\end{equation*} Hence, $\hat{T}^{\ast }(t)k=\hat{T}_{\alpha }^{\ast }(t)k\in {\K} _{\alpha C}$ that proves the statement due to continuity of the family $\hat{ T}^{\ast }(t)$.
\end{proof}

Therefore, one can consider the restriction $\hat{T}^{\odot \alpha}$
of the semigroup $\hat{T}^{\odot}$ onto $\overline{\K_{\alpha C}}$.
It will be strongly continuous semigroup with the generator
$\hat{L}^{\odot \alpha}$ which is a~restriction of ${\hat{L}}^\odot$
onto $\overline{\K_{\alpha C}}$ (see e.g. \cite[Subsection
II.2.3]{EN2000}). Hence, we have the strong solution (in the sense
of the norm in $\K_C$) to the evolution equation \eqref{ssd1} on the
linear subspace $\K_{\alpha C}$.

\begin{remark}
Let us clarify the reasons we avoid a~construction of this evolution
in $\K_C$ directly, via e.g. perturbation techniques. First of all
$(L_0,\K_{\alpha C})$ is not closed operator neither in $\K_C$ nor
in $\overline{\K_{\alpha C}}$. To make it closed, one can consider
the operator $L_0$ in $\K_C$ on its maximal domain
$\D^\ast:=\{G\in\K_C|DG\in\K_C\}$. However, this domain is not dense
in $\K_C$. Under condition of Proposition~\ref{pr3} one can show
that $\K_{\alpha C}\subset \D^\ast$, but it is not clear whether
$\D^\ast\subset \overline{\K_{\alpha C}}$. Therefore, we are not
able to work in the space $\overline{\K_{\alpha C}}$, staying on the
operator-dependent space $\overline{\D^\ast}$. Suppose one can prove
estimate like \eqref{relbound}. Then one can show that
$(\hat{L}^\ast, \D^\ast)$ will be a~generator of a~$C_0$-semigroup
$W(t)$ on $\overline{\D^\ast }$. Even in this case it seems to be
very difficult to show that this semigroup will be $\K_{\alpha
C}$-invariant.
\end{remark}

\setcounter{example}{0}
\begin{example}[revisited]
To apply Theorem~\ref{hint} to Example~\ref{ex-Gl} it is enough to
check \eqref{D-bdd} and \eqref{nusmall}. One has
\[
d(x,\xi)=\exp\Bigl\{s\sum_{y\in\xi}\phi(x-y)\Bigr\}\leq
\nu^{|\xi|},
\]
where $\nu=1$ for $s=0$ and $\nu=e^{s\bar{\phi}}\geq1$,
$\bar{\phi}=\max\limits_{x\in\X}\phi(x)$ for $s\in (0;1]$ provided
$\phi$ is bounded on $\X$. If $s=0$ then \eqref{nusmall} is true
(whenever condition \eqref{s0smallz} is satisfied). For the bounded
$\phi$ and $s\in(0;1]$ one may rewrite \eqref{nusmall} in the
following form:
\begin{equation}\label{strongersn0smallz}
e^{C\beta_{s}} + \frac{z}{C}e^{s\bar{\phi}+C\beta_{s-1}}<\frac{3}{2}.
\end{equation}
Note, that \eqref{strongersn0smallz} is the stronger version of
condition \eqref{sn0smallz}.
\end{example}

\begin{example}[revisited]
According to \eqref{smallparBDLP-1}--\eqref{smallparBDLP-2},
\begin{align*}
d(x,\xi)&=m+\varkappa^-\sum_{y\in\xi}a^-(x-y)\leq
m+A^-\varkappa^-|\xi|\\&<m+A^-\frac{m}{4C}|\xi|<m\Bigl(1+\frac{A^-}{4C}\Bigr)(1+|\xi|),
\end{align*}
where $A^-=\|a^-\|_{L^\infty(\X)}$. Therefore, \eqref{D-bdd} holds
with $\nu=1$, which makes \eqref{nusmall} obvious.
\end{example}

\subsection{Stationary equation}\label{subsect-evol-se}
In this subsection we study the question about stationary solutions
to \eqref{ssd1}. For any $s\ge0$, we consider the following subset
of $\K_C$
\[
\K_{\alpha C}^{(s)}:=\bigl\{ k\in \K_{\alpha
C}\bigm| k(\emptyset)=s\bigr\}.
\]
We define $\widetilde{\K}$ to be the closure of $\K_{\alpha
C}^{(0)}$ in the norm of $\K_C$. It is clear that $\widetilde{\K}$
with the norm of $\K_C$ is a~Banach space.

\begin{proposition}
Let \eqref{est-d}, \eqref{est-b}, and \eqref{D-bdd} be satisfied
with
\begin{equation}\label{statior-est}
a_1+\frac{a_2}{C}<2.
\end{equation}
Assume, additionally, that
\begin{equation}\label{nonezerodeath}
  d(x,\emptyset)>0, \quad x\in\X.
\end{equation}
Then for any $\alpha\in(0;\frac{1}{\nu})$ the stationary equation
\begin{equation}\label{stationaryeqn}
\hat{L}^\ast k=0
\end{equation}
has a~unique solution $k_\inv$ from $\K_{\alpha C}^{(1)}$ which is
given by the expression
\begin{equation}\label{invsol}
k_\inv=1^*+(\1-S)^{-1}E.
\end{equation}
Here $1^\ast$ denotes the function defined by
$1^\ast(\eta)=0^{|\eta|}$, $\eta\in\Ga_0$, the function
$E\in\K_{\alpha C}^{(0)}$ is such that
\[
E(\eta)=\1_{\Ga^{(1)}}(\eta)\sum_{x\in\eta}\frac{b(x,\emptyset)}{d(x,\emptyset)},
\quad \eta\in\Ga_0,
\]
and $S$ is a~generalized Kirkwood--Salzburg operator on
$\tilde{\K}$, given by
\begin{align}
\left( Sk\right) \left( \eta \right) =&-\frac{1}{D\left( \eta
\right) } \sum_{x\in \eta }\int_{\Ga _{0}\setminus \left\{ \emptyset
\right\} }k(\zeta \cup \eta )(K_{0}^{-1}d(x,\cdot \cup \eta
\setminus x))(\zeta
)d\la (\zeta ) \label{KSoper}\\
&+\frac{1}{D\left( \eta \right) }\sum_{x\in \eta }\int_{\Ga
_{0}}k(\zeta \cup (\eta \setminus x))(K_{0}^{-1}b(x,\cdot \cup \eta
\setminus x))(\zeta )d\la (\zeta ),\notag
\end{align}
for $\eta \neq \emptyset$ and $\left( Sk\right) \left( \emptyset
\right) =0$. In particular, if $b(x,\emptyset)=0$ for a.a. $x\in\X$
then this solution is such that
\begin{equation}\label{deltasol}
k_\inv^{(n)}=0, \qquad n\geq 1.
\end{equation}
\end{proposition}
\begin{remark}
It is worth noting that \eqref{est-d1}, \eqref{est-b1} imply
\eqref{nonezerodeath}.
\end{remark}
\begin{proof}
Suppose that \eqref{stationaryeqn} holds
for some $k\in\K_{\alpha C}^{(1)}$. Then
\begin{align}
D\left( \eta \right) k(\eta )=&-\sum_{x\in \eta }\int_{\Ga
_{0}\setminus \left\{ \emptyset \right\} }k(\zeta \cup \eta
)\bigl(K_{0}^{-1}d(x,\cdot \cup
\eta \setminus x)\bigr)(\zeta )d\la (\zeta ) \notag\\
&+\sum_{x\in \eta }\int_{\Ga _{0}}k(\zeta \cup (\eta \setminus
x))\bigl(K_{0}^{-1}b(x,\cdot \cup \eta \setminus x)\bigr)(\zeta
)d\la (\zeta ). \label{rel}
\end{align}
The equality \eqref{rel} is satisfied for any $k\in\K_{\alpha
C}^{(1)}$ at the point $\eta=\emptyset$. Using the fact that
$D(\emptyset)=0$ one may rewrite \eqref{rel} in terms of the
function $\tilde{k}=k-1^*\in\K_{\alpha C}^{(0)}$. Namely,
\begin{align}
D\left( \eta \right) \tilde{k}(\eta )=&-\sum_{x\in \eta }\int_{\Ga
_{0}\setminus \left\{ \emptyset \right\} }\tilde{k}(\zeta \cup \eta
)\bigl(K_{0}^{-1}d(x,\cdot \cup
\eta \setminus x)\bigr)(\zeta )d\la (\zeta ) \notag\\
&+\sum_{x\in \eta }\int_{\Ga _{0}}\tilde{k}(\zeta \cup (\eta
\setminus x))\bigl(K_{0}^{-1}b(x,\cdot \cup \eta \setminus
x)\bigr)(\zeta )d\la (\zeta ).\notag\\ &+\sum_{x\in \eta
}0^{|\eta\setminus x|}b(x,\eta\setminus x).\label{rel1}
\end{align}
As a~result,
\[
\tilde{k}(\eta)=(S\tilde{k})(\eta)+E(\eta),\quad
\eta\in\Ga_0.
\]
Next, for $\eta\neq \emptyset$
\begin{align*}
&C^{-\left\vert \eta \right\vert }\left\vert \left( Sk\right) \left( \eta
\right) \right\vert \\
\leq &\frac{C^{-\left\vert \eta \right\vert }}{D\left( \eta \right)
}\sum_{x\in \eta }\int_{\Ga _{0}\setminus \left\{ \emptyset \right\}
}\left\vert k(\zeta \cup \eta )\right\vert \left\vert
(K_{0}^{-1}d(x,\cdot \cup \eta \setminus x))(\zeta
)\right\vert d\la (\zeta ) \\
&+\frac{C^{-\left\vert \eta \right\vert }}{D\left( \eta \right)
}\sum_{x\in \eta }\int_{\Ga _{0}}k(\zeta \cup (\eta \setminus
x))\left\vert (K_{0}^{-1}b(x,\cdot \cup \eta \setminus
x))(\zeta )\right\vert d\la (\zeta ) \\
\leq &\frac{\left\Vert k\right\Vert _{\K_C}}{D\left( \eta \right)
}\sum_{x\in \eta }\int_{\Ga _{0}\setminus \left\{ \emptyset \right\}
}C^{\left\vert \zeta \right\vert }\left\vert (K_{0}^{-1}d(x,\cdot
\cup \eta \setminus
x))(\zeta )\right\vert d\la (\zeta ) \\
&+\frac{\left\Vert k\right\Vert _{\K_C}}{D\left( \eta \right)
}\frac{1}{C} \sum_{x\in \eta }\int_{\Ga _{0}}C^{\left\vert \zeta
\right\vert }\left\vert (K_{0}^{-1}b(x,\cdot \cup \eta \setminus
x))(\zeta )\right\vert
d\la (\zeta ) \\
\leq &\frac{\left\Vert k\right\Vert _{\K_C}}{D\left( \eta \right) }D\left(
\eta \right) \left( a_{1}-1+\frac{a_{2}}{C}\right)=\left( a_{1}-1+\frac{a_{2}}{C}\right)\|k\| _{\K_C}.
\end{align*}
Hence,
\[
\|S\|=a_{1}+\frac{a_{2}}{C}-1<1
\]
in $\widetilde{K}$. This finishes the proof.
\end{proof}

\begin{remark}
The name of the operator \eqref{KSoper} is motivated by
Example~\ref{ex-Gl}. Namely, if $s=0$ then the operator
\eqref{KSoper} has form
\begin{align*}
\left( Sk\right) \left( \eta \right) =\frac{1}{m| \eta| }\sum_{x\in
\eta }e_{\la}(e^{-\phi(x-\cdot)},\eta \setminus x)\int_{\Ga
_{0}}k(\zeta \cup (\eta \setminus
x))e_{\la}(e^{-\phi(x-\cdot)}-1,\zeta )d\la (\zeta ),\notag
\end{align*}
that is quite similar of the so-called Kirkwood--Salsburg operator
known in mathematical physics (see e.g. \cite{Rue1969, KK2003a}).
For $s=0$ condition \eqref{statior-est} has form
$\frac{z}{C}e^{C\beta_{-1}}<1$ (cf. \eqref{s0smallz}). Under this
condition, the stationary solution to \eqref{stationaryeqn} is a
unique and coincides with the correlation function of the Gibbs
measure, corresponding to potential $\phi$ and activity $z$.
\end{remark}

\begin{remark} It is worth pointing out that $b(x,\emptyset)=0$
in the case of Example~\ref{ex-BDLP}. Therefore, if we suppose (cf.
\eqref{smallparBDLP-1}, \eqref{smallparBDLP-2}) that
$2\varkappa^-C<m$ and $2\varkappa^+a^+(x)\leq C\varkappa^-a^-(x)$,
for $x\in\X$, condition \eqref{statior-est} will be satisfied.
However, the unique solution to \eqref{stationaryeqn} will be given
by \eqref{deltasol}. In the next example we improve this statement.
\end{remark}

\begin{example}
Let us consider the following natural modification of BDLP-model
coming from Example~\ref{ex-BDLP}: let $d$ be given by
\eqref{BDLP-d} and
\begin{equation}
b(x,\ga)=\kappa+\varkappa^+\sum_{y\in\ga}a^+(x-y), \quad x\in\X\setminus\ga,\
\ga\in\Ga,\label{BDLP-b-mod}
\end{equation}
where $\varkappa^+,a^+$ are as before and $\kappa>0$. Then, under
assumptions
\begin{equation}\label{aaa1}
 2\max\Bigl\{\varkappa^-C;\frac{2\kappa }{C}\Bigr\}<m
\end{equation}
and
\begin{equation}\label{aaa2}
2\varkappa^+a^+(x)\leq C\varkappa^-a^-(x), \quad x\in\X,
\end{equation}
we obtain for some $\delta>0$
\begin{align*}
\int_{\Ga _{0}}\left\vert K_{0}^{-1}d\left( x,\cdot \cup \xi \right)
\right\vert \left( \eta \right) C^{\left\vert \eta \right\vert
}d\la \left( \eta \right)&=d(x,\xi)+C\varkappa^-\leq \Bigl(1+\frac{1}{2+\delta}\Bigr)d(x,\xi)\\
\int_{\Ga _{0}}\left\vert K_{0}^{-1}b\left( x,\cdot \cup \xi \right)
\right\vert \left( \eta \right) C^{\left\vert \eta \right\vert }d\la
\left( \eta \right)&=b(x,\xi)+C\varkappa^+
\\&\leq \kappa
+\frac{1}{2}C\varkappa^-\sum_{y\in\xi}a^-(x-y)+\frac{m}{4}C<\frac{C}{2}d(x,\xi).
\end{align*}
The latter inequalities imply \eqref{statior-est}. In this case,
$E(\eta)=\1_{\Ga^{(1)}}(\eta)\frac{\kappa}{m}$.
\end{example}

\begin{remark}
If $a^+(x)=a^-(x)$, $x\in\X$ and $\varkappa^+=z\varkappa^-$,
$\kappa=zm$ for some $z>0$ then $b(x,\ga)=zd(x,\ga)$ and the Poisson
measure $\pi_z$ with the intensity $z$ will be symmetrizing measure
for the operator $L$. In particular, it will be invariant measure.
This fact means that its correlation function $k_z(\eta)=z^{|\eta|}$
is a~solution to \eqref{stationaryeqn}. Conditions \eqref{aaa1} and
\eqref{aaa2} in this case are equivalent to $4z< C$ and
$2\varkappa^-C<m$. As a~result, due to uniqueness of such solution,
\[
1^*(\eta)+z(\1-S)^{-1}\1_{\Ga^{(1)}}(\eta)=z^{|\eta|}, \quad
\eta\in\Ga_0.
\]
\end{remark}

\section{Scalings}\label{sect-VS}
For the reader convenience, we start from the idea of the
Vlasov-type scaling. The general scheme for the birth-and-death
dynamics as well as for the conservative ones may be found in
\cite{FKK2010a}. The realizations of this approach for the Glauber
dynamics (Example 1 with $s=0$) and for the BDLP dynamics (Example
2) were considered in \cite{FKK2010b, FKK2010c}, correspondingly.
The idea of the Vlasov-type scaling consists in the following.

We would like to construct some scaling $L_\eps$, $\eps>0$, of the
generator $L$, such that the following scheme holds. Suppose that we
have a~semigroup $\hat{U}_\eps(t)$ with the generator ${\hat{L}}
_\eps$ in some $\mathcal{L}_{C_\eps}$, $\eps>0$. Consider the dual
semigroup $\hat{U} _\eps^\ast(t)$. Let us choose an initial function
of the corresponding Cauchy problem with a~singularity in $\eps$.
Namely, $\eps^{|\eta|} k_0^{(\eps)}(\eta) \sim r_0(\eta)$,
$\eps\rightarrow 0$, $\eta\in\Ga_0$ for some function $r_0$, which
is independent of $\eps$. The scaling $L\mapsto L_\eps$ should be
chosen in such a~way that first of all the corresponding semigroup
$\hat{U}_\eps^\ast(t)$ preserves the order of the singularity:
\begin{equation}\label{ordersing}
\eps^{|\eta|} (\hat{U}_\eps^\ast(t)k_0^{(\eps)})(\eta) \sim
r_t(\eta), \ \ \eps\rightarrow 0, \ \ \eta\in\Ga_0,
\end{equation}
and, secondly, the dynamics $r_0 \mapsto r_t$ preserves the
Lebesgue--Poisson exponents. Namely, if
$r_0(\eta)=e_\la(\rho_0,\eta)$ then $r_t(\eta)=e_\la(\rho_t,\eta)$.
There exists explicit (nonlinear, in general) differential equation
for $\rho_t$:
\begin{equation}\label{V-eqn-gen}
\dfrac{\partial}{\partial t}\rho_t(x) = \upsilon(\rho_t)(x)
\end{equation}
which will be called the Vlasov-type equation.

Now we explain an informal way to realize such a~scheme. Let us
consider for any $\eps>0$ the following mapping (cf.
\eqref{isometry}) defined for functions on $\Ga_0$
\begin{equation}
(R_\eps r)(\eta):=\eps^{|\eta|} r(\eta).
\end{equation}
This mapping is ``self-dual'' with respect to the duality
\eqref{duality}, moreover, $R_\eps^{-1}=R_{\eps^{-1}}$. Having
$R_{\eps} k^{(\eps)}_0\sim r_0$, $\eps\rightarrow 0$, we need $r_t
\sim R_\eps \hat{U}_\eps^\ast(t)k_0^{(\eps)} \sim R_\eps \hat{U}
_\eps^\ast(t)R_{\eps^{-1}} r_0$, $\eps\rightarrow 0$. Therefore, we
have to show that for any $t\geq 0$ the operator family $R_\eps
\hat{U} _\eps^\ast(t)R_{\eps^{-1}}$, $\eps>0$ has limiting (in a
proper sense) operator $U(t)$ and
\begin{equation}\label{chaospreserving}
U(t)e_\la(\rho_0)=e_\la(\rho_t).
\end{equation}
But, heuristically,
$\hat{U}^\ast_\eps(t)=\exp{\{t{\hat{L}}^\ast_\eps\}}$ and $R_\eps
\hat{U}_\eps^\ast(t)R_{\eps^{-1}}=\exp{\{t R_\eps
{\hat{L}}_\eps^\ast R_{\eps^{-1}} \}}$. Let us consider the
``renormalized'' operator
\begin{equation}\label{renorm_def}
{{\hat{L}}_\ren^\ast} := R_\eps {\hat{L}}_\eps^\ast R_{\eps^{-1}}.
\end{equation}
In fact, we need that there exists an operator ${\hat{L}}_V^\ast$
such that $\exp{\{t R_\eps {\hat{L}}_\eps^\ast R_{\eps^{-1}}
\}}\rightarrow \exp{\{t{\hat{L}}_V^\ast\}=:U(t)}$ satisfying
\eqref{chaospreserving}. Therefore, an heuristic way to produce
scaling $L\mapsto L_\eps$ is to demand that
\[
\lim_{\eps\rightarrow 0}\left(\dfrac{\partial}{\partial
t}e_\la(\rho_t,\eta)-{{\hat{L}}_\ren^\ast}
e_\la(\rho_t,\eta)\right)=0, \qquad \eta\in\Ga_0
\]
provided $\rho_t$ satisfies \eqref{V-eqn-gen}. The point-wise limit
of ${{\hat{L}}_\ren^\ast}$ will be natural candidate for
${\hat{L}}_V^\ast$.

Note that \eqref{renorm_def} implies informally that
${\hat{L}}_\ren=R_{\eps^{-1}}{\hat{L}} _\eps R_\eps$. We propose
below the scheme to give rigorous meaning to the idea introduced
above. We consider, for a~proper scaling $L_\eps$, the
``renormalized'' operator ${\hat{L}}_\ren$ and prove that it is a
generator of a~strongly continuous contraction semigroup
$\hat{U}_\ren(t)$ in $\mathcal{L}_C$. Next, we show that the formal
limit ${\hat{L}}_V$ of ${\hat{L}}_\ren$ is a~generator of a~strongly
continuous contraction semigroup $\hat{U}_V(t)$ in $\mathcal{L}_C$.
Finally, we prove that $\hat{U}_\ren(t)\rightarrow\hat{U}_V(t)$
strongly in $\mathcal{L}_C$. This implies weak*-convergence of the
dual semigroups $\hat{U}_\ren^\ast(t)$ to $\hat{U}_V^\ast(t)$. We
explain also in which sense $\hat{U}_V^\ast(t)$ satisfies the
properties above.

Let us consider for any $\eps\in(0;1]$ the following scaling of
\eqref{BaDGen}
\begin{align}
(L_\eps F)(\ga ):=&\sum_{x\in \ga }d_\eps (x,\ga
\setminus x)\left[F(\ga \setminus x)-F(\ga )\right]\notag\\&+\eps ^
{-1}\int_{\R^{d}}b_\eps (x,\ga )\left[F(\ga \cup x)-F(\ga )\right]
dx, \label{BaDGenScaled}
\end{align}
and define the renormalized operator
$\hat{L}_{\eps,\mathrm{ren}}:=R_{\eps^{-1}}K^{-1}L_\eps KR_\eps$.
Using the same arguments as in the proof of
Proposition~\ref{prop_desc_oper}, we get
 \begin{align}
(\hat{L}_{\eps,\mathrm{ren}}G)(\eta )=&-\sum_{\xi \subset \eta }G(\xi )\eps^{-|\eta\setminus\xi|}\sum_{x\in \xi
}\bigl(K_0^{-1}d_\eps(x,\cdot\cup\xi\setminus x)\bigr)(\eta\setminus
\xi)\notag\\
&+\sum_{\xi \subset \eta }\int_{\R^{d}}\,G(\xi \cup
x)\eps^{-|\eta\setminus\xi|}\bigl(K_0^{
-1}b_\eps(x,\cdot\cup\xi)\bigr)(\eta\setminus \xi)
dx.\label{newexpreps}
\end{align}
Below we generalize slightly the previous introduced notations: for
$\eps\in(0;1]$, $\alpha\in(0;1)$
\begin{align*}
D_\eps(\eta)&:=\sum_{x\in\eta}d_\eps(x,\eta\setminus x);\\
\D^{(\eps)}&:=\bigl\{G\in\mathcal{L}_C\bigm|
D^{(\eps)}(\cdot)G\in\mathcal{L}_C\bigr\};\\
(L_0^{(\eps)}G)(\eta)&:=-
D_\eps(\eta)G(\eta), \qquad G\in\D^{(\eps)};\\
(L_1^{(\eps)}G)(\eta)&:=(\hat{L}_{\eps,\mathrm{ren}}G)(\eta ) -
(L_0^{(\eps)}G)(\eta), \qquad G\in\D^{(\eps)}.
\end{align*}

Suppose that there exists $a_{1}\geq1$, $a_2>0$, $A>0$, $N\in
\N_{0}$, $\nu\geq 1$ such that for all $\xi \in \Ga _{0}$, for a.a.
$x\in\X$, and for any $\eps\in(0;1]$
\begin{align}
\sum_{x\in\xi}\int_{\Ga _{0}}\left\vert K_{0}^{-1}d_\eps \left(
x,\cdot \cup \xi \setminus x\right) \right\vert \left( \eta \right)
\eps^{-|\eta|} C^{\left\vert \eta \right\vert }d\la \left( \eta
\right) \leq & a_{1} D_\eps \left( \xi \right) ,
\label{est-d-eps} \\
\sum_{x\in\xi}\int_{\Ga _{0}}\left\vert K_{0}^{-1}b_\eps \left(
x,\cdot \cup \xi \setminus x \right) \right\vert \left( \eta \right)
\eps^{-|\eta|}C^{\left\vert \eta \right\vert }d\la \left( \eta
\right) \leq &a_{2}D_\eps \left( \xi \right) ,
\label{est-b-eps}\\
d_\eps \left( x,\xi \right) \leq &\,A(1+\left\vert \xi \right\vert)
^{N}\nu^{\left\vert \xi \right\vert }. \label{D-bdd-eps}
\end{align}
Without loss of generality we will assume that all constant in
\eqref{est-d-eps}--\eqref{D-bdd-eps} are the same as before.
\begin{proposition}\label{prop_exist_ren}
\begin{enumerate}
\item Let conditions \eqref{est-d-eps} and \eqref{est-b-eps} hold
with
\begin{equation}\label{a_small_1.5}
a_1+\frac{a_2}{C}<\frac{3}{2}.
\end{equation}
Then, for any $\eps\in(0;1]$, $\bigl(\hat{L}_{\eps,\mathrm{ren}},
\D^{(\eps)}\bigr)$is a~generator of the holomorphic semigroup
$\hat{U}_\eps(t)$ on $\mathcal{L}_C$.
\item Assume, additionally, that \eqref{D-bdd-eps} is satisfied with
\begin{equation}\label{small_nu_1.5}
1\leq\nu<\frac{C}{a_2}\Bigl(\frac{3}{2}-a_1\Bigr).
\end{equation}
Then there exists $\alpha_0\in(0;\frac{1}{\nu})$ such that for any
$\alpha\in(\alpha_0;\frac{1}{\nu})$ and for any $\eps\in(0;1]$ there
exists a~strongly continuous semigroup $\hat{U}^{\odot
\alpha}_\eps(t)$ on the space $\overline{\K_{\alpha C}}$ with the
generator $\hat{L}^{\odot
\alpha}_\eps=\hat{L}_{\eps,\mathrm{ren}}^\ast$
 on the domain
 \[
 \Dom\bigl( L^{\odot \alpha}_\eps\bigr)
 = \bigl\{k\in \overline{\K_{\alpha
C}} \bigm|\hat{L}_{\eps,\mathrm{ren}}^\ast k \in
\overline{\K_{\alpha C}} \bigr\}.
\]
Note that, for $k\in\K_{\alpha
C}$
\begin{align}\label{oper_adj_eps}
(\hat{L}_{\eps,\mathrm{ren}}^\ast k)(\eta)=&-\sum_{x\in \eta
}\int_{\Ga _{0}} k(\xi \cup \eta )\eps^{-|\xi|} \bigl(
K_0^{-1}d_\eps (x,\cdot\cup \eta\setminus x)\bigr)
(\xi)d\la (\xi ) \\
& +\sum_{x\in \eta }\int_{\Ga _{0}}k(\xi \cup (\eta \setminus
x))\eps^{-|\xi|} \bigl( K_0^{-1}b_\eps (x,\cdot\cup \eta\setminus
x)\bigr) (\xi)d\la (\xi).\notag
\end{align}
\end{enumerate}
\end{proposition}
\begin{proof}
1. Identically to the proof of Lemma~\ref{lem1} we show that
$\bigl(L_0^{(\eps)}, \D^{(\eps)})\in\mathcal{H}_C(\omega)$ for any
$\omega\in(0;\frac{\pi}{2})$. Next, in the same way as in the proof
of Lemma~\ref{lem2} we prove that, for any $\Re z>0$,
\begin{equation}\label{expr_need}
\bigl\Vert L_{1}^{(\eps)}R(z,L_{0}^{(\eps)})\bigr\Vert \leq a_1-1+\frac{a_2}{C}<\frac{1}{2},
\end{equation}
since \eqref{statior-est} is satisfied. Note also that we may show
also another bound (cf.~\eqref{deep}):
\begin{equation}
\| L_1^{(\eps)}G\|<\frac{1}{2}\|L_0^{(\eps)}G\|,\quad
G\in\mathcal{L}_C.
\end{equation}
 Hence, one can
prove the statement in the same way as Theorem~\ref{th1}.

2. Similarly to Proposition~\ref{pr3}, we obtain that, under
condition \eqref{D-bdd-eps}, $\K_{\alpha C}\subset
\Dom(\hat{L}_{\eps,\mathrm{ren}}^\ast )$ for any
$\alpha\in(0;\frac{1}{\nu})$. Using \eqref{small_nu_1.5}, we are
able to choose $\theta\in(a_1+\frac{a_2\nu}{C};\frac{3}{2})$. Then
\eqref{nusmall} is satisfied, and
$\alpha_0:=\frac{a_{2}}{C(\theta-a_{1})}\in(0;\frac{1}{\nu})$. The
same considerations as in Theorem~\ref{hint} finish the proof.
\end{proof}

\begin{assumption}
For all $\eta,\xi\in\Ga_0$ and a.a. $x\in\X$ the following limits
exist and coincide:
\begin{align}
\lim_{\eps\rightarrow0} \eps^{-|\eta|}\bigl( K_{0}^{-1}d_\eps \left(
x,\cdot \cup \xi \right) \bigr) \left( \eta \right) =
\lim_{\eps\rightarrow0} \eps^{-|\eta|}\bigl( K_{0}^{-1}d_\eps \left(
x,\cdot \right) \bigr) \left( \eta \right) =:
D_x^V(\eta);\label{d-lim}
\\
\lim_{\eps\rightarrow0} \eps^{-|\eta|}\bigl( K_{0}^{-1}b_\eps \left(
x,\cdot \cup \xi \right) \bigr) \left( \eta \right) =
\lim_{\eps\rightarrow0} \eps^{-|\eta|}\bigl( K_{0}^{-1}b_\eps \left(
x,\cdot \right) \bigr) \left( \eta \right) =:
B_x^V(\eta).\label{b-lim}
\end{align}
We would like to emphasize, that above limits should not depend on
$\xi$. The collection of examples for such $d_\eps$, $b_\eps$ can be
found in \cite{FKK2010a}. Note that \eqref{d-lim}, \eqref{b-lim}
imply, in particular,
\begin{equation}
\lim_{\eps\rightarrow 0}d_\eps(x,\xi)=D_x^V(\emptyset),
\qquad
\lim_{\eps\rightarrow
0}b_\eps(x,\xi)=B_x^V(\emptyset),\label{lim_empty}
\end{equation}
for all $\xi\in\Ga_0$ and a.a. $x\in\X$.
\end{assumption}

Combining condition \eqref{d-lim} with \eqref{b-lim}, we have
point-wise limit for $\hat{L}_{\eps,\mathrm{ren}}$:
\begin{equation}
(\hat{L}_VG)(\eta ):=-\sum_{\xi \subset \eta }G(\xi )\sum_{x\in \xi
}D_{x}^{V}(\eta\setminus \xi)+\sum_{\xi \subset \eta }\int_\X\,G(\xi
\cup x)B_{x}^{V}(\eta\setminus \xi) dx.\label{newexprV}
\end{equation}
Set
\begin{align*}
D_V(\eta)&:=\sum_{x\in\eta}D_{x}^V(\emptyset);\\
\D^V&:=\bigl\{G\in\mathcal{L}_C\bigm|
D_V(\cdot)G\in\mathcal{L}_C\bigr\};\\
(L_0^VG)(\eta)&:=-
D_V(\eta)G(\eta), \qquad G\in\D^V;\\
(L_1^VG)(\eta)&:=(\hat{L}_{V}G)(\eta ) - (L_0^V G)(\eta), \qquad
G\in\D^V.
\end{align*}
Suppose that for a.a. $x\in\X$
\begin{align}
\int_{\Ga _{0}}\big| D_x^V(\eta)\bigr| C^{\left\vert \eta
\right\vert }d\la \left( \eta \right) &\leq a_{1}
D_x^{V}(\emptyset),
\label{est-d-V} \\
\int_{\Ga _{0}}\big| B_x^V(\eta)\bigr| C^{\left\vert \eta
\right\vert }d\la \left( \eta \right) &\leq a_{2}D_x^{V}(\emptyset)
,
\label{est-b-V}\\
D_x^{V}(\emptyset)&\leq A, \label{D-bdd-V}
\end{align}
where the constants are the same as before.

\begin{remark}\label{suffcond}
It is worth pointing out that conditions
\eqref{est-d-V}--\eqref{D-bdd-V}, in general, are weaker than
\eqref{est-d-eps}--\eqref{D-bdd-eps}. Indeed, if
$b_\eps(x,\ga)=b'(x,\ga)+\eps\cdot b''(x,\ga)$ then \eqref{est-b-V}
is an assumption on function $b'$ only, whereas \eqref{est-b-eps}
requires additional conditions on $b''$.
\end{remark}

Let $c>0$. We define $\bar{B}_c^\infty$ to be the closed ball of
radius $c$ in the Banach space $L^\infty(\X)$.

\begin{proposition}\label{prop_exist_V}
\begin{enumerate}
\item Let conditions \eqref{est-d-V}, \eqref{est-b-V}, and
\eqref{a_small_1.5} hold. Then $\bigl(\hat{L}_{V}, \D^V\bigr)$ is
a~generator of the holomorphic semigroup $\hat{U}_V(t)$ on
$\mathcal{L}_C$.
\item Suppose, additionally, \eqref{D-bdd-V}
is satisfied. Then, there exists $\alpha_0\in(0;1)$ such that for
any $\alpha\in(\alpha_0;1)$ there exists a~strongly continuous
semigroup $\hat{U}^{\odot \alpha}_V(t)$ on the space
$\overline{\K_{\alpha C}}$ with the generator $\hat{L}^{\odot
\alpha}_V=\hat{L}_{V}^\ast$,
\[
\Dom\bigl( L^{\odot \alpha}_V\bigr) = \bigl\{k\in
\overline{\K_{\alpha C}} \bigm|\hat{L}_V^\ast k \in
\overline{\K_{\alpha C}} \bigr\}.
\]
Moreover, for $k\in\K_{\alpha C}$
\begin{align}\label{oper_adj_V}
(\hat{L}_{V}^\ast k)(\eta)=&-\sum_{x\in \eta }\int_{\Ga _{0}} k(\xi
\cup \eta )D_x^V(\xi)d\la (\xi ) \\
& +\sum_{x\in \eta }\int_{\Ga _{0}}k(\xi \cup (\eta \setminus
x))B_x^V(\xi)d\la (\xi).\notag
\end{align}
\item Let $\alpha\in(\alpha_0;1)$, $\rho_0\in\bar{B}_{\alpha C}^\infty$. Then the evolution equation
\begin{equation}\label{QFP}
\begin{cases}
\dfrac{\partial}{\partial t} k_t=\hat{L}^\ast_V
k_t\\
k_t\bigr|_{t=0}=e_\la(\rho_{0},\eta)
\end{cases}
\end{equation}
has a~unique solution $k_t=e_\la(\rho_t)$ in $\overline{\K_{\alpha
C}}$ provided $\rho_t$ belongs to $\bar{B}_{\alpha C}^\infty$ and
satisfies the Vlasov-type equation
\begin{align} \label{Vlasov_eqn}
\dfrac{\partial}{\partial t} \rho_t(x)=&-\rho_t(x)\int_{\Ga _{0}}
e_\la(\rho_t,\xi)D_x^V(\xi)d\la (\xi ) \\&+\int_{\Ga
_{0}}e_\la(\rho_t,\xi)B_x^V(\xi)d\la (\xi). \notag
\end{align}

\end{enumerate}
\end{proposition}
\begin{proof}
1. The proof for the first statement is similar to the analogous one
one in Proposition~\ref{prop_exist_ren}.

2. The same arguments as for the proof of Proposition~\ref{pr3} show
that, for any $\alpha\in(0;1)$, $\K_{\alpha C}\subset
\Dom(\hat{L}^*_V)$. Next, by \eqref{a_small_1.5}, let us now take
$\theta\in(a_1+\frac{a_2}{C};\frac{3}{2})$. Then we can set
$\alpha_0:=\frac{a_{2}}{C(\theta-a_{1})}\in(0;1)$. The second
statement can be handled now in much the same way as in
Theorem~\ref{hint}.

3. Since $\rho_0\in\bar{B}_{\alpha C}^\infty$ implies
$k_0\in\overline{\K_{\alpha C}}$ then the Cauchy problem \eqref{QFP}
has a~unique solution in $\overline{\K_{\alpha C}}$. On the other
hand, according to \eqref{oper_adj_eps}, for any $\rho_t\in
\bar{B}_{\alpha C}^\infty$
\begin{align}\notag
(\hat{L}_{V}^\ast
e_\la(\rho_{t}))(\eta)=&-\sum_{x\in \eta }e_\la(\rho_{t},\eta)\int_{\Ga _{0}} e_\la(\rho_{t},\xi)D_x^V(\xi)d\la (\xi ) \\
& +\sum_{x\in \eta }e_\la(\rho_{t},\eta\setminus x)\int_{\Ga
_{0}}e_\la(\rho_{t},\xi)B_x^V(\xi)d\la (\xi).\label{actonexp}
\end{align}
Combining \eqref{actonexp} with the equality
\[
\dfrac{\partial}{\partial t}
e_\la(\rho_t,\eta)=\sum_{x\in\eta}\rho_t(x)
e_\la(\rho_t,\eta\setminus x),
\]
we can assert that $k_t=e_\la(\rho_t)$ is a~solution to \eqref{QFP},
with $\rho_t$ given by \eqref{Vlasov_eqn}.
\end{proof}

\begin{remark}
The question about existence and uniqueness of solutions to the
Vlasov-type equation \eqref{Vlasov_eqn} in some ball
$\bar{B}_{\alpha C}^\infty$ of $L^\infty(\X)$ shall be solved
separately in each concrete model, see e.g. \cite{FKK2010b,
FKK2010c}.
\end{remark}

Our next goal is to study the question about convergence of the
semigroups $\hat{U}_\eps(t)$ to $\hat{U}_V(t)$ in $\mathcal{L}_C$.

We begin by proving the following abstract statement.
\begin{lemma}\label{lemma_conv}
Let $X$ be a~Banach space, and let $ \left(
A_\eps,\mathfrak{D}_\eps\right) $, $\left( B_{\eps
},\mathfrak{D}_\eps \right) $, $\eps \geq 0$ be closed, densely
defined operators on $X$. Suppose that there exists $\beta >0$ and
$z\in\mathbb{C}$ with $\Re z>\beta$ such that $z \in \rho \left(
A_\eps\right) $ for all $\eps\geq0$ and
\begin{gather}
\kappa:=\sup_{\eps >0} \bigl\Vert \left( A_\eps-z \1 \right)
^{-1}\bigr\Vert<\infty,\label{unibound1}\\
\sigma :=\sup_{\eps \geq 0}\left\Vert B_{\eps }\left(
A_\eps-z\1 \right) ^{-1}\right\Vert <1, \label{unibound2} \\
\left( A_\eps-z \1 \right)
^{-1}\overset{s}{\longrightarrow } \left( A_0-z\1 \right)
^{-1}, \quad \eps\rightarrow0,\label{convres1}\\
B_{\eps }\left( A_\eps-z\1 \right) ^{-1}\overset{s}{\longrightarrow
} B_{0}\left( A_{0}-z\1 \right) ^{-1},\quad\eps \rightarrow 0.
\label{convgood}
\end{gather} Then $z$ belongs to the resolvent set of $L_{\eps }:=A_\eps+B_{\eps }$,
$\eps \geq 0$ and
\[
\left( L_{\eps }-z\1\right) ^{-1}\overset{s}{\longrightarrow }
\left( L_{0}-z\1 \right) ^{-1}, \quad \eps \rightarrow 0.
\]
\end{lemma}
\begin{proof}
 For any
$\eps\geq 0$ we set $C_\eps:=( A_\eps-z\1 ) ^{-1}$, then we have
$\mathrm{Ran}(C_\eps) =\Dom\left( A_\eps\right) =\Dom\left(
B_\eps\right) = \Dom(L_\eps)=\mathfrak{D}_\eps $. Therefore, for any
$z \in \rho \left( A_\eps\right) $ one can write
\begin{equation*}
L_\eps -z\1=A_\eps+B_\eps-z\1 =\bigl( B_\eps\left( A_\eps-z\1\right)
^{-1}+\1\bigr) \left( A_\eps-z \1\right). \label{o1}
\end{equation*}
By \eqref{unibound2}, the operator $B_\eps\left( A_\eps-z \1\right)
^{-1}+\1=B_\eps C_\eps+\1$ is invertible with bounded inverse
$D_\eps$. Moreover,
\begin{equation}\label{normD}
\|D_\eps\|\leq\frac{1}{1-\|B_\eps C_\eps\|}\leq\frac{1}{1-\sigma}.
\end{equation}
Therefore, we have
that $z\in\rho(L_\eps)$ and
\begin{equation}
(L_\eps -z\1)^{-1}=\left( A_\eps-z\1\right)^{-1}\bigl( B_\eps
C_\eps+\1\bigr)^{-1}=C_\eps D_\eps . \label{expr_for_res}
\end{equation}
Next,
\begin{align*}
&D_\eps-D_0=\left( B_{\eps }C_{\eps}+\1\right)
^{-1}-\left( B_{0}C_{0}+\1\right) ^{-1} \\
=&\left( B_{\eps }C_{\eps}+\1\right) ^{-1}\left( \left(
B_{0}C_{0}+\1\right) -\left( B_{\eps }C_{\eps}+\1\right) \right)
\left(
B_{0}C_{0}+\1\right) ^{-1} \\
=&D_{\eps}\left( B_{0}C_{0}-B_{\eps }C_{\eps}\right) D_{0},
\end{align*} thus, according to \eqref{normD} and \eqref{convgood},
for any $x\in X$
\begin{align*}
\|D_\eps x-D_0 x\| &\leq \|D_\eps\|\cdot\|\left( B_{0}C_{0}-B_{\eps
}C_{\eps}\right) D_{0}x\|\\&\leq \frac{1}{1-\sigma}\|\left(
B_{0}C_{0}-B_{\eps }C_{\eps}\right) D_{0}x\|\rightarrow 0, \quad
\eps\rightarrow 0.
\end{align*}
Hence, $D_\eps\overset{s}{\longrightarrow} D_0$. Then, using
\eqref{expr_for_res} and \eqref{convres1}, we have for any $x\in X$
\begin{align*}
&\bigl\|(L_\eps-z\1)^{-1}x-(L_0-z\1)^{-1}x\bigr\|\\=&\,\bigl\|C_\eps
D_\eps x-C_0D_0x\bigr\|=\bigl\|C_\eps(D_\eps-D_0)x+(C_\eps-C_0)D_0x\bigr\|\\
\leq&\, \bigl\|C_\eps\bigr\|\cdot\bigl\|(D_\eps-D_0)x\bigr\|+\bigl\|
(C_\eps -C_0)D_0 x\bigr\|\\\leq &\, \kappa
\cdot\bigl\|(D_\eps-D_0)x\bigr\|+\bigl\|
(C_\eps -C_0)D_0 x\bigr\|\rightarrow 0, \quad \eps\rightarrow 0.
\end{align*}
The statement is proven.
\end{proof}

Now we are able to prove result about convergence in
$\mathcal{L}_C$.

\begin{theorem}\label{conv_qo}
Let conditions \eqref{est-d-eps}, \eqref{est-b-eps}, and
\eqref{a_small_1.5} are satisfied. Suppose that convergences
\eqref{d-lim}, \eqref{b-lim} take place for all $\eta\in\Ga_0$ as
well as in the sense of $\mathcal{L}_C$. Assume also that there
exists $\sigma>0$ such that (cf. \eqref{lim_empty}) either
\begin{equation}\label{either}
d_\eps(x,\xi)\leq \sigma D_x^V(\emptyset)\qquad \text{or} \qquad
d_\eps(x,\xi)\geq \sigma D_x^V(\emptyset)
\end{equation}
is satisfied for all $\xi\in\Ga_0$ and for a.a. $x\in\X$. Then
$\hat{U}_\eps(t)\overset{s}{\longrightarrow } \hat{U}_V(t)$ in
$\mathcal{L}_C$ uniformly on finite time intervals.
\end{theorem}
\begin{proof} First of all note that $\mathcal{L}_C$-convergence
in \eqref{d-lim}, \eqref{b-lim} together with \eqref{lim_empty}
yields \eqref{est-d-V}, \eqref{est-b-V} provided \eqref{est-d-eps},
\eqref{est-b-eps} hold. Then, by Propositions~\ref{prop_exist_ren},
\ref{prop_exist_V}, the semigroups $\hat{U}_\eps(t)$, $\hat{U}_V(t)$
exist in $\mathcal{L}_C$. To prove their convergence it is enough to
show the strong convergence of the resolvent corresponding to the
generators of this semigroup, see e.g.
\cite[Theorem~III.4.8]{EN2000}. To verify this, we apply
Lemma~\ref{lemma_conv} taking $A_\eps=L_0^{(\eps)}$,
$B_\eps=L_1^{(\eps)}$, $L_\eps=\hat{L}_{\eps,\mathrm{ren}}$,
$\mathfrak{D}_0=\D^V$, $\mathfrak{D}_\eps=\D^{(\eps)}$, $\eps>0$.
Below we check the conditions of this lemma.

Let us fix any $z>0$. It is easily seen that \eqref{unibound1} is
satisfied since
\[
\bigl\Vert \bigl( L_0^{(\eps)}-z \1 \bigr) ^{-1}\bigr\Vert\leq
\frac{1}{z }
\]
for all $\eps\in(0;1]$. Clearly, \eqref{expr_need} implies
\eqref{unibound2}. Let $G\in\mathcal{L}_C$. Then
\begin{align*}
&\bigl\| \bigl( L_0^{(\eps)}-z \1 \bigr)
^{-1}G- \bigl( L_0^{V}-z \1 \bigr)
^{-1}G\bigr\|_C \\
\leq &\, \int_{\Ga_0}\frac{\bigl|D^{(\eps)}(\eta)-D^V(\eta)\bigr|}
{\bigl(z+D^V(\eta)\bigr)\bigl(z+D^{(\eps)}(\eta)\bigr)}|G(\eta)|d\la(\eta).
\end{align*}
By \eqref{lim_empty}, for all $ \eta\in\Ga_0$
\begin{equation}\label{eqn1}
D^{(\eps)}(\eta)\rightarrow D^V(\eta), \quad \eps\rightarrow0.
\end{equation}
Then the
inequality
\[
\frac{\bigl|D^{(\eps)}(\eta)-D^V(\eta)\bigr|}
{\bigl(z+D^V(\eta)\bigr)\bigl(z+D^{(\eps)}(\eta)\bigr)}\leq\frac{1}{z+D^V(\eta)}+\frac{1}{z+D^{(\eps)}(\eta)}\leq\frac{2}{z}
\]
implies \eqref{convres1} by the dominated
convergence theorem.

Let inequality $d_\eps(x,\xi)\leq \sigma D_x^V(\emptyset)$ hold for
all $\xi\in\Ga_0$ and a.a. $x\in\X$. Then, by Lemma~\ref{Minlos},
\begin{align}
&\Bigl\|L_1^{(\eps)} \bigl( L_0^{(\eps)}-z \1 \bigr)
^{-1}G- L_1^{V} \bigl( L_0^{V}-z \1 \bigr)
^{-1}G\Bigr\|_C \notag\\
\leq &\, \Bigl\|\bigl(L_1^{(\eps)} - L_1^{V}\bigr)\bigl( L_0^{V}-z
\1 \bigr) ^{-1}G \Bigr\|_C+\Bigl\|L_1^{(\eps)}\Bigl(\bigl(
L_0^{(\eps)}-z \1 \bigr) ^{-1}-\bigl( L_0^{V}-z \1 \bigr)
^{-1}\Bigr)G \Bigr\|_C\notag\\\leq&\,
\int_{\Ga_0}\frac{|G(\xi)|}{z+D^{V}(\xi)}\sum_{x\in\xi}\int_{\Ga_0}
\Bigl| \eps^{-|\eta|} K_0^{-1}d_\eps(x,\cdot\cup\xi\setminus
x)(\eta)-D_{x}^{V}(\eta)\Bigr|C^{|\eta|}d\la(\eta)C^{|\xi|}d\la(\xi)
\notag\\&
+\frac{1}{C}\int_{\Ga_0}\frac{|G(\xi)|}{z+D^{V}(\xi)}\sum_{x\in\xi}
\int_{\Ga_0}\Bigl| \eps^{-|\eta|}
K_0^{-1}b_\eps(x,\cdot\cup\xi\setminus
x)(\eta)-B_{x}^{V}(\eta)\Bigr|C^{|\eta|}d\la(\eta)C^{|\xi|}d\la(\xi)\notag\\&+
\int_{\Ga_0}|G(\xi)|\frac{\bigl|D^{(\eps)}(\xi)-D^V(\xi)\bigr|}
{\bigl(z+D^V(\xi)\bigr)\bigl(z+D^{(\eps)}(\xi)\bigr)}\sum_{x\in\xi}
\int_{\Ga_0}\eps^{-|\eta|}\Bigl( \bigl|
K_0^{-1}d_\eps(x,\cdot\cup\xi\setminus
x)(\eta)\bigr|\notag\\&\qquad\qquad+\frac{1}{C}\bigl|
K_0^{-1}b_\eps(x,\cdot\cup\xi\setminus
x)(\eta)\bigr|\Bigr)C^{|\eta|}d\la(\eta)C^{|\xi|}d\la(\xi).\label{newdopnew}
\end{align}
Convergence in $\mathcal{L}_C$ for \eqref{d-lim}, \eqref{b-lim}
together with \eqref{eqn1} implies that all three integrand
functions of $\xi$ appearing in \eqref{newdopnew} converge to $0$
$\la$-a.s., as $\eps\rightarrow 0$. To use dominated convergence
theorem we will show that the following functions are uniformly
bounded. Using \eqref{est-d-eps}, \eqref{est-d-V}, and
\eqref{either}, we get
\begin{align*}
&\frac{1}{z+D^V(\xi)}\sum_{x\in\xi}\int_{\Ga_0}
\Bigl| \eps^{-|\eta|} K_0^{-1}d_\eps(x,\cdot\cup\xi\setminus
x)(\eta)-D_{x}^{V}(\eta)\Bigr|C^{|\eta|}d\la(\eta)\\
\leq &
\frac{a_1}{z+D^V(\xi)}\sum_{x\in\xi}\bigl(d_\eps(x,\xi)+D_x^V(\emptyset)\bigr)
\leq
\frac{a_{1}(1+\sigma)}{z+D^V(\xi)}\sum_{x\in\xi}D_x^V(\emptyset)\leq
a_{1}(1+\sigma).
\end{align*}
Analogously, by \eqref{est-b-eps}, \eqref{est-b-V}, and
\eqref{either},
\[
\frac{1}{z+D^V(\xi)}\sum_{x\in\xi}\int_{\Ga_0} \Bigl| \eps^{-|\eta|}
K_0^{-1}b_\eps(x,\cdot\cup\xi\setminus
x)(\eta)-B_{x}^{V}(\eta)\Bigr|C^{|\eta|}d\la(\eta)\leq
a_{1}(1+\sigma).
\]
According to \eqref{est-d-eps}, \eqref{est-b-eps}, and
\eqref{either},
\begin{align*}
&\frac{\bigl|D^{(\eps)}(\xi)-D^V(\xi)\bigr|}
{\bigl(z+D^V(\xi)\bigr)\bigl(z+D^{(\eps)}(\xi)\bigr)}\sum_{x\in\xi}
\int_{\Ga_0}\eps^{-|\eta|}\Bigl( \bigl|
K_0^{-1}d_\eps(x,\cdot\cup\xi\setminus
x)(\eta)\bigr|\\&\qquad\qquad+\frac{1}{C}\bigl|
K_0^{-1}b_\eps(x,\cdot\cup\xi\setminus
x)(\eta)\bigr|\Bigr)C^{|\eta|}d\la(\eta)\\\leq&
\frac{D^{(\eps)}(\xi)+D^V(\xi)}
{\bigl(z+D^V(\xi)\bigr)\bigl(z+D^{(\eps)}(\xi)\bigr)}\sum_{x\in\xi}\Bigl(a_1d_\eps(x,\xi)+\frac{a_2}{C}d_\eps(x,\xi)\Bigr)\\
\leq&\frac{D^{(\eps)}(\xi)}
{\bigl(z+D^V(\xi)\bigr)\bigl(z+D^{(\eps)}(\xi)\bigr)}\Bigl(a_1+\frac{a_2}{C}\Bigr)\sigma D^{V}(\xi)\\&+\frac{D^V(\xi)}
{\bigl(z+D^V(\xi)\bigr)\bigl(z+D^{(\eps)}(\xi)\bigr)}\Bigl(a_1+\frac{a_2}{C}\Bigr)D^{(\eps)}(\xi)\leq\Bigl(a_1+\frac{a_2}{C}\Bigr)(1+\sigma).
\end{align*}
Hence, \eqref{convgood} is proved.

In the case $d_\eps(x,\xi)\geq \sigma D_x^V(\emptyset)$,
$\xi\in\Ga_0$, we rewrite l.h.s. of \eqref{convgood} in a~another
manner. Namely,
\begin{align*}
&\Bigl\|L_1^{(\eps)} \bigl( L_0^{(\eps)}-z \1 \bigr)
^{-1}G- L_1^{V} \bigl( L_0^{V}-z \1 \bigr)
^{-1}G\Bigr\|_C \\
\leq &\, \Bigl\|\bigl(L_1^{(\eps)} - L_1^{V}\bigr)\bigl(
L_0^{\eps}-z \1 \bigr) ^{-1}G \Bigr\|_C+\Bigl\|L_1^{V}\Bigl(\bigl(
L_0^{(\eps)}-z \1 \bigr) ^{-1}-\bigl( L_0^{V}-z \1 \bigr)
^{-1}\Bigr)G \Bigr\|_C\\\leq&\,
\int_{\Ga_0}\frac{|G(\xi)|}{z+D^{(\eps)}(\xi)}\sum_{x\in\xi}\int_{\Ga_0}
\Bigl| \eps^{-|\eta|} K_0^{-1}d_\eps(x,\cdot\cup\xi\setminus
x)(\eta)-D_{x}^{V}(\eta)\Bigr|C^{|\eta|}d\la(\eta)C^{|\xi|}d\la(\xi)\\&
+\frac{1}{C}\int_{\Ga_0}\frac{|G(\xi)|}{z+D^{(\eps)}(\xi)}\sum_{x\in\xi}
\int_{\Ga_0}\Bigl| \eps^{-|\eta|}
K_0^{-1}b_\eps(x,\cdot\cup\xi\setminus
x)(\eta)-B_{x}^{V}(\eta)\Bigr|C^{|\eta|}d\la(\eta)C^{|\xi|}d\la(\xi)\\&+
\int_{\Ga_0}|G(\xi)|\frac{\bigl|D^{(\eps)}(\xi)-D^V(\xi)\bigr|}
{\bigl(z+D^V(\xi)\bigr)\bigl(z+D^{(\eps)}(\xi)\bigr)}\sum_{x\in\xi}
\int_{\Ga_0}\Bigl(D_{x}^{V}(\eta)
+\frac{1}{C}B_{x}^{V}(\eta)\Bigr)C^{|\eta|}d\la(\eta)C^{|\xi|}d\la(\xi).
\end{align*}
Repeating all estimates done for the first alternative of
\eqref{either} we get the desired result.
\end{proof}

\begin{remark}
Note that in all examples considered in \cite{FKK2010a} the function
$d_\eps(x,\xi)$ is monotone in $\eps$. Taking into account
\eqref{lim_empty}, condition \eqref{either} becomes natural.
\end{remark}

\setcounter{example}{0}
\begin{example}[revisited]
Let us consider for $\eps\in[0;1]$, $s\in[0;1]$
\[
d_\eps(x,\ga)=\exp\Bigl\{\eps s\sum_{y\in\ga\setminus
x}\phi(x-y)\Bigr\}, \quad
b_\eps(x,\ga)=z\exp\Bigl\{\eps(s-1)\sum_{y\in\ga}\phi(x-y)\Bigr\},
\]
Analogously to the previous computations,
\begin{align*}
\int_{\Ga _{0}}\left\vert K_{0}^{-1}d_\eps\left( x,\cdot \cup \xi
\right) \right\vert \left( \eta \right) \eps^{-|\eta|}C^{\left\vert
\eta \right\vert
}d\la \left( \eta \right)&=d_\eps(x,\xi)e^{C\eps^{-1}\beta_{\eps s}}\\
\int_{\Ga _{0}}\left\vert K_{0}^{-1}b_\eps\left( x,\cdot \cup \xi
\right) \right\vert \left( \eta \right) \eps^{-|\eta|}C^{\left\vert
\eta \right\vert }d\la \left( \eta
\right)&=b_\eps(x,\xi)e^{C\eps^{-1}\beta_{\eps (s-1)}}\\&\leq
zd_\eps(x,\xi)e^{C\eps^{-1}\beta_{\eps (s-1)}},
\end{align*}
since $\phi\geq0$. Let $s\in(0;1]$. Suppose that $\tilde{\beta}:=
\int_\X \phi(x) e^{\phi(x)} dx<\infty$. Then for $\tau\in[-1;1]$,
$\eps\in[0,1]$
\[
\eps^{-1}\beta_{\eps\tau}\leq\eps^{-1}\int_\X \eps
|\tau|\phi(x)\sup_{\tau\in[-1,1]} e^{\eps \tau\phi(x)}
dx\leq\tilde{\beta}.
\]
The bound \eqref{a_small_1.5} will be proved once we show
$e^{C\tilde{\beta}}\bigl(1+\frac{z}{C}\bigr)<\frac{3}{2}$. If $s=0$
then, similarly, we need $\beta:=\int_\X\phi(x)dx<\infty$ and
$\frac{z}{C}e^{C\beta}<\frac{1}{2}$. Note also that the conditions
$\beta<\infty$ and $\bar{\phi}=\sup_\X\phi(x)<\infty$ yield
$\tilde{\beta}\leq e^{\bar{\phi}}\beta<\infty$. For the case $s=0$
condition \eqref{D-bdd-eps} holds automatically. If $s\in(0;1]$ one
should assume $\bar{\phi}<\infty$ then $\nu=e^{s\bar{\phi}}$
(uniformly by $\eps\in(0;1]$). Then to guarantee
\eqref{small_nu_1.5} we need
$e^{C\tilde{\beta}}\bigl(1+\frac{z}{C}e^{s\bar{\phi}}\bigr)<\frac{3}{2}$.
Therefore, under such conditions we obtain statement of
Proposition~\ref{prop_exist_ren}. Next,
\begin{align}
 \lim_{\eps\rightarrow0} \eps^{-|\eta|}\bigl( K_{0}^{-1}d_\eps \left( x,\cdot \cup \xi \right)
\bigr) \left( \eta \right) &= \lim_{\eps\rightarrow0} \exp\Bigl\{\eps s\sum_{y\in\xi}\phi(x-y)\Bigr\}e_\la\left(\frac{e^{\eps s \phi(x-\cdot)}-1}{\eps},\eta\right)\notag\\
&= e_\la\bigl(s \phi(x-\cdot),\eta\bigr)=:D_x^V(\eta);\label{bbb1}
\end{align}
and, analogously,
\begin{equation}
 \lim_{\eps\rightarrow0} \eps^{-|\eta|}\bigl( K_{0}^{-1}b_\eps \left( x,\cdot \cup \xi \right)
\bigr) \left( \eta \right) = ze_\la\bigl((s-1)
\phi(x-\cdot),\eta\bigr)=: B_x^V(\eta). \label{bbb2}
\end{equation}
Since $D_x^V(\emptyset)=1\leq d_\eps(x,\eta)$, the second
alternative of \eqref{either} is satisfied. In order to use
Proposition~\ref{prop_exist_V} and Theorem~\ref{conv_qo} we need to
verify the convergences \eqref{bbb1} and \eqref{bbb2} in
$\mathcal{L}_C$ (recall that this implies \eqref{est-d-V} and
\eqref{est-b-V}, see the proof of Theorem~\ref{conv_qo}). To do this
let us note that for any $\tau\in[-1;1]$
\begin{align*}
&\biggl|\exp\Bigl\{\eps \tau\sum_{y\in\xi}\phi(x-y)\Bigr\}e_\la\left(\frac{e^{\eps \tau \phi(x-\cdot)}-1}{\eps},\eta\right)
- e_\la\bigl(\tau \phi(x-\cdot),\eta\bigr)\biggr|\\\leq&\max\biggl\{
\exp\Bigl\{ \tau\sum_{y\in\xi}\phi(x-y)\Bigr\},1\biggr\}e_\la\left(\frac{\bigl|e^{\eps \tau \phi(x-\cdot)}-1\bigr|}{\eps},\eta\right)+e_\la\bigl(|\tau| \phi(x-\cdot),\eta\bigr)\\\leq&\left(\max\biggl\{
\exp\Bigl\{ \tau\sum_{y\in\xi}\phi(x-y)\Bigr\},1\biggr\}+1\right)e_\la\bigl( \phi(x-\cdot),\eta\bigr),
\end{align*}
and the last function of $\eta$ belongs to $\mathcal{L}_C$ for all
$\xi\in\Ga_0$ and a.a. $x\in\X$ provided $\phi\in L^1(\X)$. By
\eqref{intexp}, the Vlasov equation \eqref{Vlasov_eqn} now has the
following form
\[
\dfrac{\partial}{\partial t} \rho_t(x)=-\rho_t(x)\exp\bigl\{ s(\rho_t\ast
\phi)(x)\bigr\} +z\exp\bigl\{ (s-1)(\rho_t\ast
\phi)(x)\bigr\}.
\]
Here and below $\ast$ means usual convolution of functions in $\X$.
\end{example}

\begin{example}[revisited]
Let $d_\eps(x,\ga\setminus x)=m +
\eps\varkappa^-\sum_{y\in\ga\setminus x}a^-(x-y)$,
$b_\eps(x,\ga)=\eps\varkappa^+\sum_{y\in\ga}a^+(x-y)$. Comparing
with the previous notations we have changed $\varkappa^\pm$ onto
$\eps\varkappa^\pm$. Clearly, conditions \eqref{smallparBDLP-1},
\eqref{smallparBDLP-2} implies the same inequalities for
$\eps\varkappa^\pm$. Note also that $d_\eps$ is decreasing in
$\eps\rightarrow0$. Therefore, to apply all results of this section
to BDLP-model we should prove the convergence \eqref{d-lim},
\eqref{b-lim} in $\mathcal{L}_C$. Note, that
\begin{align*}
\eps^{-|\eta|}K_{0}^{-1}d_\eps\left( x,\cdot \cup \xi \right) \left(
\eta \right)&=d_\eps(x,\xi)\eps^{-|\eta|}0^{|\eta|}+
\eps\eps^{-|\eta|}\varkappa^-\1_{\Ga^{(1)}}(\eta)\sum_{y\in\eta}a^{-}(x-y)\\
&=d_\eps(x,\xi)0^{|\eta|}+\1_{\Ga^{(1)}}(\eta)\sum_{y\in\eta}a^{-}(x-y)
\\&\rightarrow
m0^{|\eta|}+\1_{\Ga^{(1)}}(\eta)\sum_{y\in\eta}a^{-}(x-y)=:D_x^V(\eta)
\end{align*}
and, analogously,
\begin{align*}
\eps^{-|\eta|}K_{0}^{-1}b_\eps\left( x,\cdot \cup \xi \right)
 \left( \eta \right)&=b_\eps(x,\xi)0^{|\eta|}+
\1_{\Ga^{(1)}}(\eta)\sum_{y\in\eta}a^{+}(x-y)\\&\rightarrow
\1_{\Ga^{(1)}}(\eta)\sum_{y\in\eta}a^{+}(x-y)=:B_x^V(\eta).
\end{align*}
The convergence in $\mathcal{L}_C$ is obvious now. The Vlasov
equation has the following form
\[
\dfrac{\partial}{\partial t} \rho_t(x) = \varkappa^{+}(a^{+}\ast\rho_t)(x)- \varkappa^{-}\rho_{t}(x)(a^{-}\ast \rho_{t})(x)-
m\rho_{t}(x).
\]
The existence and uniqueness of the solution to this equation was
studied in \cite{FKK2010c}.
\end{example}

\begin{remark}
By duality \eqref{duality}, Theorem~\ref{conv_qo} yields
weak*-convergence of the semigroups $\hat{U}^{\odot \alpha}_\eps(t)$
to $\hat{U}^{\odot \alpha}_V(t)$ in $\overline{\K_{\alpha C}}$. To
prove such convergence in the strong sense we need additional
analysis of their generators. The problem concerns the fact that we
have explicit expression for the generator $\hat{L}^{\odot
\alpha}_V=\hat{L}_{V}^\ast$ only on the core $\bigl\{k\in
{\K_{\alpha C}} \bigm|\hat{L}_V^\ast k \in \overline{\K_{\alpha C}}
\bigr\}$. However, we are able to show such convergence for the
Glauber dynamics described in Example 1 for $s=0$ using modified
technique (see \cite{FKK2010b}).
\end{remark}

\section*{Acknowledgements}
The financial support of DFG through the
SFB 701 (Bielefeld University) and German-Ukrainian Project 436 UKR
113/97 is gratefully acknowledged.

\end{document}